%% file: Main_ArXiv.tex
\documentclass[aps,prf,reprint,amsmath,amssymb]{revtex4-2}

\usepackage{graphicx}% Include figure files
\usepackage{dcolumn}% Align table columns on decimal point
\usepackage{bm}% bold math
\usepackage{import}
\import{}{Packages.tex}

\import{}{Commands.tex}
\import{}{Acronyms.tex}

\begin{document}

\preprint{APS/123-QED}

% \title{Variational Optimization for Distributed and Boundary Control of Stochastic Fields}% Force line breaks with \\
%TC:ignore
\title{Leveraging Stochasticity for Open Loop and Model Predictive Control of Complex Fluid Systems}

\author{George I. Boutselis}%
\thanks{Authors contributed equally. Authors appear in last name order.}%
\author{Ethan N. Evans}%
\thanks{Authors contributed equally} \thanks{Corresponding author. email: eevans89@gmail.com}%

\author{Marcus A. Pereira}%
\thanks{Authors contributed equally}  \altaffiliation[Also at ]{Georgia Institute of Technology, Institute of Robotics and Intelligent Machines}

\author{Evangelos A. Theodorou}
 \altaffiliation[Also at ]{Georgia Institute of Technology, Institute of Robotics and Intelligent Machines}
%  \thanks{Corresponding author. email: evangelos.theodorou@gatech.edu}
\affiliation{%
  Georgia Institute of Technology Department of Aerospace Engineering, \\
  Atlanta, GA, USA 30332
}%

\date{\today}% It is always \today, today,

\begin{abstract}
Stochastic Spatio-Temporal processes are prevalent across domains ranging from modeling of plasma to the turbulence in fluids to the wave function of quantum systems. This letter studies a measure-theoretic description of such systems by describing them as evolutionary processes on Hilbert spaces, and in doing so, derives a framework for spatio-temporal manipulation from fundamental thermodynamic principles. This approach yields a variational optimization framework for controlling stochastic fields. The resulting scheme is applicable to a wide class of spatio-temporal processes and can be used for optimizing parameterized control policies. Our simulated experiments explore the application of two forms of this approach on four stochastic spatio-temporal processes, with results that suggest new perspectives and directions for studying stochastic control problems for spatio-temporal systems.
\end{abstract}

% \begin{description}
% Secondary publications and information retrieval purposes.
% You may use the \texttt{description} environment to structure your abstract;
% use the optional argument of the \verb+\item+ command to give the category of each item. 
% \end{description}
% \end{abstract}

\maketitle
%TC:endignore

\section{Introduction and Related Work}
Many complex systems in nature vary spatially and temporally, and are often represented as Stochastic Partial Differential Equations (SPDEs). These systems are ubiquitous in nature and engineering, and can be found in fields such as applied physics, robotics, autonomy, and finance~\cite{chow2007stochastic,da1992stochastic,StochasticNavierStokes_2004,SPDEs_Neural2017,Pardouxt_1980,PhysRevE_Bang1994,RAMA_SPDEs_Economics_2005,Belavkin_ControlV1_2005,Belavkin_ControlV2_2005}. Examples of stochastic spatio-temporal processes include the \textit{Poisson-Vlassov} equation in plasma physics, the Heat, \textit{Burgers} and \textit{Navier-Stokes} equations in fluid mechanics, the \textit{Zakai} and \textit{Belavkin} equations in classical and quantum filtering. Despite their ubiquity and significance to many areas of science and engineering, algorithms for stochastic control of such systems are scarce.

% Stochastic fields are systems represented by stochastic partial differential equations (SPDEs) and can be found in various applications ranging from applied physics to robotics and autonomy  \cite{chow2007stochastic,da1992stochastic,StochasticNavierStokes_2004,SPDEs_Neural2017,Pardouxt_1980,PhysRevE_Bang1994,RAMA_SPDEs_Economics_2005,Belavkin_ControlV1_2005,Belavkin_ControlV2_2005}. The \textit{Poisson-Vlassov} equation in plasma physics, the Heat, \textit{Burgers} and \textit{Navier-Stokes} equations in fluid mechanics, the \textit{Zakai} and \textit{Belavkin} SPDEs in classical and quantum filtering are just some examples of infinite-dimensional stochastic systems. Despite their ubiquity and significance to many areas of science and engineering, there is very little work on algorithms for stochastic control of such systems.

The challenges of controlling SPDEs include significant control signal time-delays, dramatic under-actuation, high dimensionality, regular bifurcations, and multi-modal instabilities. For many SPDEs, existence and uniqueness of solutions remains an open problem, and when solutions exist, they often have a weak notion of differentiability if at all. Their performance analysis must be treated with functional calculus, and their state vectors are often most conveniently described by vectors in an infinite-dimensional time-indexed Hilbert space, even for scalar 1-dimensional SPDEs. These and other challenges together represent a large subset of the current-day challenges facing the fluid dynamics and automatic control communities, and present difficulties in the development of mathematically consistent and numerically realizable algorithms.
 
The majority  of computational stochastic control methods in the literature have been dedicated to finite-dimensional systems. Algorithms for decision making under uncertainty of such systems typically rely on standard optimality principles from the Stochastic Optimal Control (SOC) literature, namely the Dynamic Programming (or Bellman) principle and the stochastic Pontryagin Maximum principle~\cite{Pontryagin1962,Bellman1964,yong1999stochastic}. The resulting algorithms typically require solving the Hamilton-Jacobi-Bellman (HJB) equation; a backward nonlinear partial differential equation (PDE) of which solutions are not scalable to high dimensional spaces. 

\begin{figure}[t!]
    \includegraphics[width=0.49\textwidth]{ctrl_arch6.png}
    \caption{Overview of architecture for the control of spatio-temporal stochastic systems, where $\rd W_j^r$ denotes a Cylindrical Wiener process at time step $j$ for simulated system rollout $r$ and  $X_{0:T}^r$ denotes the state trajectory of simulated system rollout $r$. See  \cref{eq:Iterative_optimalvariation1,eq:J_i} and related explanations for a more complete explanation.}
    \label{fig:architecture}
    \vspace{-1em}
\end{figure}

Several works (e.g. \cite{Christofides2009, GOMES201733} for the Kuramoto-Sivashinsky SPDE) propose model predictive control based methodologies for reduced order models of SPDEs based on SOC princples. These reduced order methods transform the original SPDE into a finite set of coupled Stochastic Differential Equations (SDEs). In SDE control, probabilistic representations of the HJB PDE can solve scalability via sampling techniques~\cite{Pardoux_Book2014,Fleming2006} including iterative sampling and/or parallelizable implementations~\cite{Exarchos_2017, williams2017model}. These methods have been explored in a reinforcement learning context for SPDEs \cite{Evans2019IDVRL,evans2020spatio,evans2021stochastic}.
 
Recently, a growing body of work considers deterministic PDEs, and utilize finite dimensional machine learning methods such as Deep Neural Network surrogate models which utilize standard SOC-based methodologies. In the context of fluid systems, these approaches are increasingly widespread in the literature~\cite{bieker2019deep,nair2019cluster,mohan2018deep,morton2018deep,rabault2019artificial}. A critical issue in applying controllers that rely on a limited number of modes is that they can produce concerning emergent phenomena, including spillover instabilities \cite{curtain1986robust, 1101798} and failing latent space stabilizability conditions \cite{morton2018deep}.

Outside the large body of finite dimensional methods for PDEs and/or SPDEs are a few works that attempt to extend the classical HJB theory for systems described by SPDEs. These are comprehensively explored in \cite{fabbri} and include both distributed and boundary control problems. Most notably, \cite{DaPrato1999} investigates explicit solutions to the HJB equation for the stochastic Burgers equation based on an exponential transformation, and \cite{feng2006} provides an extension of the large deviation theory to infinite dimensional spaces that creates connections to HJB theory. These and most other works on HJB theory for SPDEs mainly focus on theoretical contributions and leave literature with algorithms and numerical results tremendously sparse. Furthermore, HJB theory for boundary control has certain mathematical difficulties which impose limitations.

Alternative methodologies are derived using information theoretic control. The basis of a subset of these methods is a relation between \textit{Free energy} and \textit{Relative Entropy} in statistical physics given by
\begin{equation}\label{eq:Free_Energy_Relative_Entropy}
\text{Free Energy} \leq \text{Work} - \text{Temperature} \times \text{Entropy} 
\end{equation}
This inequality is an instantiation of the second law in stochastic thermodynamics: Increase in \textit{Entropy} results in minimizing the right hand side of the expression. In finite dimensions, connections between \cref{eq:Free_Energy_Relative_Entropy} and Dynamic Programming motivate these methods. Essentially, there exist two different points of view on decision making under uncertainty that overlap for fairly general classes of stochastic systems, as depicted in \cref{Fig:Connections}. 

These connections are extended to infinite-dimensional spaces \cite{theodorou2018linearly} (see also the Supplemental Material \cref{supsec:Connections_DP}) and are leveraged in this letter to develop practical algorithms for distributed and boundary control of stochastic fields. Specifically, we develop a generic framework for control of  stochastic fields that are  modeled  as semi-linear SPDEs. We show that optimal control of SPDEs can be casted as a variational optimization problem and then solved using sampling of infinite dimensional diffusion processes. The resulting variational optimization algorithm can be used in either fixed or receding time horizon formats for distributed and boundary control of semilinear SPDEs and utilizes adaptive importance sampling of stochastic fields. The derivation relies on a non-trivial generalization of stochastic calculus to arbitrary Hilbert spaces and has broad applicability.

\begin{figure}[t!]
  \begin{center}
    \includegraphics[width=0.48\textwidth]{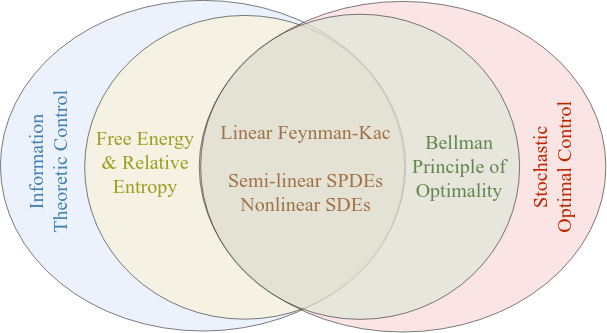}
  \end{center}
  \vspace{-10pt}
  \captionof{figure}{\label{Fig:Connections}Connection between the free energy-relative entropy approach and stochastic Bellman Principle of Optimality. }
\end{figure}

This manuscript presents an open loop and model predictive control methodology for control of SPDEs related to fluid dynamics which are grounded on the theory of stochastic calculus in function spaces, which is not restricted to any particular finite representation of the original system. The control updates are independent of the method used to numerically simulate the SPDEs, which allows the most suitable problem dependent numerical scheme (e.g., finite differences, Galerkin methods, finite elements, etc.) to be employed. 

Furthermore, deriving the variational optimization approach for optimal control entirely in Hilbert spaces overcomes  numerical issues, including matrix singularities and SPDE space-time noise degeneracies that typically arise in finite dimensional representations of SPDEs. Thus, the work in this letter is a generalization of information theoretic control methods in finite dimensions \cite{todorov2009efficient,TheodorouCDC2012,entropy_2015,Kappen2005b} to infinite dimensions and inherits crucial characteristics from its finite dimensional counterparts. 
%the ability to deal with non-quadratic cost functions and nonlinear dynamics.

However, the primary benefit of the information theoretic approach presented in this work is that the stochasticity inherent in the system can be \textit{leveraged} for control. Namely, The inherent system stochasticity is utilized for exploration in the space of trajectories of SPDEs in Hilbert spaces, which provide a Newton-type parameter update on the parametrized control policy, as shown in \cref{fig:architecture}. Importance sampling techniques are incorporated to iteratively guide the sampling distribution, and result in a mathematically consistent and numerically realizable sampling-based algorithm for distributed and boundary control of semi-linear SPDEs.

\section{Preliminaries and Problem Formulation}

%We would like to emphasize here that the prior work on stochastic control of SDEs using the exponential transformation of the value function can be split into discrete \cite{todorov2009efficient} and continuous time  formulations \cite{Kappen2005b}. The information theoretic approach to stochastic control was  introduced in Reference \cite{DaiPra1996}  and then further explored in terms of algorithms in robotics, autonomy and control in references \cite{Grady_MPPI2015,Grady_MPPI2016,Grady_MPPI2018,Grady_ICRA_17,williams2017model,TRO_GRady_2018,entropy_2015,TheodorouCDC2012}.  In this article, we show that from all these different approaches to stochastic control under the value function exponential transformation, the information theoretic approach generalizes in a very elegant fashion because of its pure measure-theoretic nature. This allows us to derive sampling-based controllers without relying on model reduction techniques; an approach that is typically used for computational control of fields (\textbf{cite}, \textbf{also use arguments from rebattle phase of NIPS}). 
%%%%%%%%%%%%%%%%%%%%%%%%%%TABLE%%%%%%%%%%%%%%%%%%%%%%
\begin{table*}[t!]
    \captionof{table}{\label{tab:semilinear_pdes} Examples of commonly known semi-linear PDEs in a \textit{fields representation} with subscript $x$ representing partial derivative with respect to spatial dimensions and subscript $t$ representing partial derivatives with respect to time. The associated operators $\A$ and $F(t,X)$ in the Hilbert space formulation are colored blue and violet, respectively, and specified in separate columns for clarity.}
    \begin{ruledtabular}
    \begin{tabular}{  l  l  l  l  l }
        \textbf{Equation Name} & \textbf{Partial Differential Equation}
         & \textbf{Linear Op.} $\A $  & \textbf{Non-linear Op.} $F(t,X)$  & \textbf{Field State} \\ 
        \hline
        Nagumo & $u_t = {\color{blue}\epsilon u_{xx}} + {\color{violet} u(1-u)(u-\alpha)}$ & $u_{xx}$ & $u(1-u)(u-\alpha)$ & Voltage\\

        Heat & $u_t = {\color{blue}\epsilon u_{xx}}$ & $u_{xx}$ &  & Heat/temperature\\  
        
        Burgers (viscous) & $u_t = {\color{blue}\epsilon u_{xx}} - {\color{violet}u u_x}$ & $u_{xx}$ & $uu_x$ & Velocity\\
        
        Allen-Cahn & $u_t = {\color{blue}\epsilon u_{xx} } + {\color{blue}u} - {\color{violet} u^3}$ & $u_{xx}+u$ & $u^3$ & Phase of a material \\

        Navier-Stokes & $u_t = {\color{blue}\epsilon \Delta u} - \nabla p - {\color{violet}(u \cdot \nabla)u}$ & $\Delta u$ & $(u \cdot \nabla)u$ & Velocity\\        

        Nonlinear Schrodinger & $u_t = {\color{blue}\frac{1}{2} i u_{xx}} + {\color{violet}i |u|^2 u=0}$ & $i u_{xx}$ & $i |u|^2 u$ &  Wavefunction\\

        Korteweg-de Vries & $u_t = - {\color{blue} u_{xxxx}} - {\color{violet} 6uu_x}$ & $u_{xxxx}$ & $uu_x$ & Plasma wave \\

        Kuramoto-Sivashinsky & $u_t = - {\color{blue}  u_{xxxx}} - {\color{blue}  u_{xx}} - {\color{violet}uu_x} $ & $u_{xx} + u_{xxxx}$ & $uu_x$ & Flame front \\
    \end{tabular}
    \end{ruledtabular}
\end{table*}
%%%%%%%%%%%%%%%%%%%%%%TABLE%%%%%%%%%%%%%%%%%%%%%%

At the core of our method are comparisons between sampled stochastic paths used to perform Newton-type control updates, as depicted in \cref{fig:architecture}. Let, $H$, $U$ be separable Hilbert spaces with inner products $\langle \cdot, \cdot \rangle_H$ and $\langle \cdot, \cdot \rangle_U$ resp,, $\sigma$-fields $\calB(H)$ and $\calB(U)$ resp. and  probability space $(\Omega, \F, \Pb) $ with filtration $\F_{t},$ $t\in [0, T]$. Consider the controlled and uncontrolled infinite-dimensional stochastic systems of the form
\begin{align}
\rd X &= \A X \rd t + F(t, X) \rd t +   \frac{1}{\sqrt{\rho}} G(t, X) \rd W(t), \label{eq:SPDEs_NoControl}\\
\rd X &= \A X\rd t + F(t, X )\rd t \nonumber\\
&\quad+ G(t, X)\bigg(\calU^{(i)}(t, X; \vtheta)\rd t+ \frac{1}{\sqrt{\rho}} \rd  W(t)\bigg), \label{eq:SPDEs_Control}
\end{align}
\noindent where  $X(0) $ is an  $\F_{0}-$measurable, $H-$valued random variable, and $\A : D(\A ) \subset H \to H $ is a linear operator, where $D(\A )$ denotes here the domain of $\A$. $F: H \to H$ and $ G : U \to H  $ are nonlinear operators that satisfy properly formulated Lipschitz conditions associated with the existence and uniqueness of solutions to \cref{eq:SPDEs_NoControl} as described in \cite[Theorem 7.2]{da1992stochastic}. The term $\calU^{(i)}(t,X; \vtheta)$ is a control operator on Hilbert space $H$ parameterized by a finite set of decision variables $\vtheta$.
 
The term  $ W(t)\in U $ corresponds to a \textit{Hilbert space Wiener process}, which is a generalization of the Wiener process in finite dimensions. When this noise profile is spatially uncorrelated, we call it a {\it cylindrical Wiener process}, which requires the added assumptions on $\A$ in \cite[Hypthesis 7.2]{da1992stochastic} in order to form a contractive, unitary, linear semigroup, which is required to guarantee existence and uniqueness of $\calF_t$-adapted weak solutions $X(t), t \geq 0$. A thorough description of the Wiener process in Hilbert spaces, along with its various forms can be found in the Supplemental Material \cref{supsec:wiener}.
 
In what follows, $\langle\cdot,\cdot\rangle_{S}$ denotes the inner product in a Hilbert space $S$ and $C([0,T];H)$ denotes the space of continuous processes in $H$ for $t\in [0,T]$. 
% Define a new process $\hat{W})(t) = W(t) - \int_0^t G^{-1}(t,X)\tilde{B}(t,X)$, and 
Define the measure on the path space of uncontrolled trajectories produced by \cref{eq:SPDEs_NoControl} as $\calL$ and define the measure on the path space of controlled trajectories produced by \cref{eq:SPDEs_Control} as $\calL^{(i)}$. The notation $\Eb_\calL$ denotes expectations over paths as Feynman path integrals.

% The connection between  \textit{U-valued $Q$}-Wiener processes and \textit{real-valued} brownian motion goes in  both ways. In particular, starting from a collection of independent  brownian motions  $  \{\beta_{j}\}_{i=1}^{\infty} $, and $ (Q, {e_{i},\lambda_{i}}) $  one can construct a  $\calQ$-Wiener processes  $ W(t) $  using  \eqref{eq:Q_Wiener}. In addition, given a $Q$-Wiener process and the corresponding eigenvalue-eigenvector pairs $({e_{i},\lambda_{i}})$ one can reconstruct independent brownian motions using \eqref{eq:Brownian_Motion}.

Many physical and engineering systems can be written in the abstract form of \cref{eq:SPDEs_NoControl} by properly defining operators $\A$, $F$ and $G$ along with their corresponding domains. Examples can be found in our simulated experiments, as well as \cref{tab:semilinear_pdes}, with more complete descriptions in \cite[Chapter 13]{da1992stochastic}). The goal of this work is to establish control methodologies for stochastic versions of such systems.

Control tasks defined over SPDEs typically quantify task completion by a measurable functional $J:C([0,T];H) \rightarrow \Rb$ referred to as the cost functional, given by
\begin{equation}\label{eq:State_cost}
    J\big(X(\cdot,\omega)\big) =  \phi\big(X(T),T\big) + \int_{t}^{T} \ell\big(X(s),s\big) \rd s,
\end{equation}
where $X(\cdot,\omega)\in C([0,T];H)$ denotes the entire state trajectory, $\phi\big(X(T),T\big)$ is a terminal state cost and $\ell\big(X(s),s\big)$ is a state cost accumulated over the time horizon $s \in [t,T] $. With this, we define the terms of \cref{eq:Free_Energy_Relative_Entropy}. More information can be found in the Supplemental Material \cref{supsec:Free_Energy_Relative_Entropy}.

Define the \textit{Free energy} of cost function $J(X)$ with respect to uncontrolled path measure $\calL$ and temperature $\rho \in \Rb$ as \cite{theodorou2018linearly}
\begin{equation}
V(X)  := -\frac{1}{\rho} \ln \Eb_{\calL} \Big[ \exp \big(-\rho J(X) \big)  \Big].
\end{equation}
Also, the \textit{Generalized Entropy} of controlled path measure $\calL^{(i)} $ with respect uncontrolled path measure $\calL$ is defined as
\begin{equation}
     S\Big( \tilde{\calL} || \calL \Big) := \left\{
\begin{array}{l l}
  -\int_{\Omega}   \frac{    \rd \calL^{(i)} }{\rd \calL}  \ln \frac{ \rd \calL^{(i)} }{\rd \calL }         \rd \calL,
  \mbox{if $\calL^{(i)} <<\calL $},  \\
  +\infty,  \quad \mbox{otherwise}, \end{array} \right.
\end{equation}
where ``$<<$'' denotes absolute continuity \cite{theodorou2018linearly}.

The relationship between free energy and relative entropy was extended to a Hilbert space formulation in \cite{theodorou2018linearly}. Based on the free energy and generalized entropy definitions,  \cref{eq:Free_Energy_Relative_Entropy} with temperature $T = \frac{1}{\rho}$ becomes the so-called Legendre transformation, and takes the form
    \begin{align} \label{eq:Legendre}
    & - \frac{1}{\rho}   \ln \Eb_{\calL} \bigg[ \exp( -\rho {J} )  \bigg]  \leq \bigg[    \Eb_{\calL^{(i)} }\left({J} \right)  -\frac{1}{\rho} S \left( \calL^{(i)} \hspace{0.05cm} \big|\big|  \calL \right)  \bigg],
    \end{align}
with  equilibrium probability measure in the form of a Gibbs distribution
\begin{equation}\label{eq:Gibbs}
\rd \calL^{*}  = \frac{\exp( - \rho J) \rd \calL }{\int_{\Omega} \exp( - \rho J)  \rd \calL },
\end{equation}
Optimality of $\calL^*$ is verified in \cite{ theodorou2018linearly}. The statistical physics interpretation of inequality \cref{eq:Legendre} is that maximization of entropy results in reduction of the available energy. At the thermodynamic equilibrium the entropy reaches its maximum and $V  = E  - T S$.

The free energy-relative entropy relation provides an elegant methodology to derive  novel algorithms for distributed and boundary control problems of SDPEs. This relation is also significant in the context of SOC literature, wherein optimality of control solutions rely on fundamental principles of optimality such as Pontryagin Maximum Principle \cite{Pontryagin1962} or the Bellman Principle of Optimality \cite{Bellman1964}. The Supplemental Material \cref{supsec:Connections_DP} shows that by applying a properly defined Feynman-Kac argument, the free energy is equivalent to a value function that satisfies the HJB equation. This connection is valid for general probability measures, including  measures defined on path spaces induced by infinite-dimensional stochastic systems.

Our derivation is general in the context of \cite{DaPrato1999}, wherein they apply a transformation that is only possible for state-dependent cost functions. The proof given in the Supplemental Material \cref{supbsec:Feynman-Kac} is novel for a generic state and time dependent cost to the best knowledge of the authors. The observation that the Legendre transformation in \cref{eq:Legendre} is connected to optimality principles from SOC motivates the use of \cref{eq:Gibbs} for the development of stochastic control algorithms. 
  
Flexibility of this approach is apparent in the context of stochastic boundary control problems, which are theoretically more challenging due to the unbounded nature of the solutions \cite{Stability_SPDEs,fabbri}. The HJB theory for these settings is not as mature and results are restricted to simplistic cases \cite{ErgodicControl_HeatSPDE}. Nonetheless, since \cref{eq:Legendre} holds for arbitrary measures, the difficulties of related works are overcome by the proposed information theoretic approach. Hence, in either the stochastic boundary control or distributed control case the free energy represents a lower bound of a {\it state cost} plus the associated {\it control effort}. Despite losing connections to optimality principles in systems with boundary control, our strategy in both distributed and boundary control settings is to optimize the {\it distance} between our parameterized control policies and the optimal measure in \cref{eq:Gibbs}, so that the lower bound of the total cost can be approached by the controlled system. Specifically, we look for a finite set of decision variables $\vtheta^*$ that yield a Hilbert space control input $\calU(\cdot)$ that minimizes the distance to the optimal path measure
 \begin{align}
    \vtheta^{*} &=   \argmax_{\vtheta}  S\big(\calL^{*} || \calL^{(i)}  \big)  \label{eq:theta} \\
  &= \argmax_\vtheta   \bigg[ -  \int_{\Omega}   \frac{    \rd \calL^{*}  }{\rd  \calL^{(i)}  }  \ln  \frac{ \rd \calL^{*} }{\rd \calL^{(i)}  }  \rd  \calL^{(i)}   \bigg]. \label{eq:theta_expand}
\end{align}

\section{Stochastic Optimization in Hilbert Spaces}

To optimize \cref{eq:theta}, we apply the chain rule for the Radon-Nikodym derivative twice~\footnote{ While this is necessary on the right term for our control update, this is applied to the left term for importance sampling, which enhances algorithmic convergence.}, which has the form
\begin{equation}
 \frac{ \rd \calL^{*} }{\rd \calL^{(i)}  }   =  \frac{ \rd \calL^{*} }{\rd \calL }   \frac{ \rd \calL }{\rd \calL^{(i)}  }. \label{eq:RN_chain}
\end{equation}

Note that the first derivative is given by \cref{eq:Gibbs} while the second derivative is given by a change of measure between control and uncontrolled infinite dimensional stochastic dynamics. This change of measure arises from a version of Girsanov's Theorem, provided with a proof in the Supplemental Material \cref{supsec:girsanov}. Under the open-loop parameterization
\begin{equation}\label{eq:ParameterControl1}
\calU(t)(\vx) = \sum_{\ell=1}^{N} m_{\ell}(\vx) u_{\ell}(t) =\vm(\vx)^{\top} \vu(t;\vtheta),
\end{equation}
Girsanov's theorem yields the following change of measure between the two SPDEs
\begin{equation}
\label{eq:radon_param}
\begin{split}
 \frac{ \rd \calL}{ \rd \calL^{(i)}  }=\exp\hspace{-0.07cm}\bigg(\hspace{-0.1cm}-\hspace{-0.05cm}\sqrt{\rho}\int_{0}^{T}\hspace{-0.2cm}\vu(t)^{\top}\bar{ \vm}(t)\hspace{-0.02cm}+\hspace{-0.02cm}\frac{\rho}{2}\int_{0}^{T}\hspace{-0.2cm}\vu(t)^{\top}\vM\vu(t)\rd t\bigg),
\end{split}
\end{equation}
with
\begin{equation}
\label{small_m}
\bar{ \vm}(t):=\bigg[\langle m_{1},\rd W(t)\rangle_{U_{0}},...,\langle m_{N},\rd W(t)\rangle_{U_{0}}\bigg]^{\top}\in\mathbb{R}^{N},
\end{equation}
\begin{equation}
\label{big_M}
\vM\in\mathbb{R}^{N\times N},\quad (\vM)_{ij}:=\langle m_{i},m_{j}\rangle_{U},
\end{equation}
where $\mathbf{x}\in\mathbb{R}^{n}$ denotes the spatial component of the SPDEs and $ m_{\ell}\in U $ are design  functions that specify how actuation is incorporated into the infinite dimensional dynamical system. This parameterization can be used for both open loop trajectory optimization as well as for model predictive control. In our experiments we apply model predictive control through re-optimization, and turn \cref{eq:ParameterControl1} into an implicit feedback type control. Optimization  using  \cref{eq:theta} with policies that explicitly depend on the stochastic field is also possible and is considered using gradient-based optimization in   \cite{Evans2019IDVRL,evans2020spatio,evans2021stochastic}. 

To simplify the optimization in \cref{eq:theta}, we further parameterize $\vu(t;\theta)$ as a simple measurable function. In this case, the parameters $\vtheta$ consist of all step functions $\{\vu_i\}$. With this representation, we arrive at our main result--an importance sampled variational controller of the form
\begin{widetext}
\begin{align}
\vu_{j}{}^{(i+1)} &=  \vu_{j}{}^{(i)}  +   \frac{1}{\sqrt{\rho} \Delta t}   \vM^{-1}  \Eb_{\calL^{(i)}}   \bigg[   \frac{\exp( - \rho J^{(i)}) }{  \Eb_{\calL^{(i)}} [\exp( - \rho  J^{(i)}] } \int_{t_j}^{t_{j+1}}\bar{ \vm}^{(i)}(t)    \bigg], \label{eq:Iterative_optimalvariation1} \\
\text{where } J^{(i)} &:= J + \frac{1}{\sqrt{\rho}}\sum_{j=1}^{L}\vu_j^{(i)\top}\int_{t_j}^{t_{j+1}}\bar{ \vm}^{(i)}(t) + \frac{\Delta t}{2}\sum_{j=1}^{L}\vu_j^{(i)\top}\vM\vu_j^{(i)},\label{eq:J_i}
\end{align}
\end{widetext}
where $\bar{\vm}^{(i)}(t)$ is defined similar to \cref{small_m}, but with the Wiener process $W^{(i)}(t) := W(t) - \sqrt{\rho} \int_0^t  \calU^{(i)}(s) \rd s$.
%with the control path dependent  function   $ \zeta^{(i)} = \zeta(\calU^{(i)}): [0,T] \times \calO \to \Rb  $ defined as

The intermediate steps of \cref{eq:Iterative_optimalvariation1} can be found in the Supplemental Material \cref{supsec:derivation}. For the purposes of implementation, we perform the approximation
% we use the following decomposition of cylindrical Wiener processes:
%  \begin{equation}\label{eq:Q_Wiener}
%     W(t) = \sum_{j=1}^{\infty} \sqrt{\lambda_{j}} \beta_{j}(t) e_{j},
%  \end{equation}
% where  $ \{\beta_{j}(t)\}  $  are real-valued Brownian motions that are mutually independent on $ (\Omega, \F, P)$ and $\{e_j\}$ form a complete orthonormal system in $U$. With this, we perform the approximation:
 \begin{equation}\label{eq:optimalvariation_1}
\int_{t_j}^{t_{j+1}}\langle m_{l},\rd W(t)\rangle_{U_{0}}  \approx  \sum_{s=1}^{R}\langle m_{l}, e_{s}\rangle_{U}\Delta\beta^{(i)}_{s}(t_j),
 \end{equation}
where $\Delta\beta^{(i)}_{s}(t_j)$ are Brownian motions sampled from the zero-mean Gaussian distribution $\Delta\beta^{(i)}_{s}(t_j)\sim\mathcal{N}(0, \Delta t)$, and $\{e_j\}$ form a complete orthonormal system in $U$. This is based on truncation of the cylindrical Wiener noise expansion
 \begin{equation}\label{eq:Q_Wiener}
    W(t) = \sum_{j=1}^{\infty} \beta_{j}(t) e_{j}.
 \end{equation}

We note that the control of SPDEs with cylindrical Wiener noise, as above, can be extended to the case in \cite{StochasticBurgers_1999}, in which $G(t,X)$ is treated as a trace-class covariance operator $\sqrt{Q}$ of a $Q$-Wiener process $\rd W_Q(t)$. See the Supplemental Material \cref{supsec:Q_Wiener_derivation} for more details. The resulting iterative control policy is identical to \cref{eq:Iterative_optimalvariation1} derived above.

%%%%%%%%%%%%%%%%%%SECTION CHANGE%%%%%%%%%%%%%%%%%%%%%%%%%
% \section{Discussion}
%  \label{sec:discussion} 
%%%%%%%%%%%%%%%%%%%%%%%%%%%%%%%%%%%%%%%%%%%%%%%%%%%%%%%%%
\section{Comparisons to Finite-Dimensional Optimization}

In light of recent work that apply finite dimensional control after reducing the SPDE model to a set of SDEs or ODEs, we highlight critical advantages of optimizing in Hilbert spaces before discretizating. The main challenge with performing optimization based control after discretization is that SPDEs typically reduce to degenerate diffusion process for which importance sampling schemes are difficult. Consider the finite dimensional SDE representation of \cref{eq:SPDEs_NoControl}
\begin{equation}\label{eq:Finite_DimensionalSDE}
\begin{split}
\rd \hat{X} &= \mathcal{A} \hat{X} \rd t + \mathcal{F}(t,\hat{X})\rd t \\
&\quad + \mathcal{G}(t,\hat{X})\bigg( \mathcal{M} \vu(t;\vtheta) \rd t + \frac{1}{\sqrt{\rho}} \mathcal{R}\rd \vbeta(t) \bigg),
\end{split}
\end{equation}
where  $ \hat{X} \in \mathbb{R}^{n} $ is an n-dimensional vector comprising of the values of the stochastic field at particular basis elements. The terms $\mathcal{A}$, $\mathcal{F}$, and $\mathcal{G}$ are matrices associated with their respective Hilbert space operators. The matrix $\mathcal{M} \in \mathbb{R}^{n \times k}$, where $k $ is the number of actuators placed in the field. The vector $\rd \vbeta \in \mathbb{R}^m$ collects noise terms and $\mathcal{R}$ collects associated finite dimensional basis vectors of \cref{eq:Q_Wiener}. The matrix $\mathcal{R}\in \Rb^{n \times m}$ is composed of $n$ rows, which is the number of basis elements used to spatially discretize the SPDE \cref{eq:SPDEs_NoControl}, and $m$ columns, which is the number of expansion terms of \cref{eq:Q_Wiener} that are used.

Girsanov's theorem for SDEs of the form \cref{eq:Finite_DimensionalSDE} requires the matrix $\calR$ to be invertible, as seen in the resulting change of measure
\begin{equation}
\begin{split}
    \frac{\rd \calL}{\rd \calL^{(i)}} &= \exp\bigg( -\sqrt{\rho} \int_0^T \big\langle \calR^{-1} \calM \vu(s,\vtheta), \rd W(s) \big \rangle_U \\
    &\;\qquad + \frac{\rho}{2} \int_0^T \big\langle \calR^{-1}\calM \vu(s,\vtheta), \calR^{-1}\calM \vu(s,\vtheta) \big\rangle_U \rd s \bigg)
\end{split}
\end{equation}
Deriving the optimal control in the finite dimensional space requires that a)  the noise term is expanded to at least as many terms as the points on the spatial discretization $n \leq m $, and b)  the resulting diffusion  matrix  $ \mathcal{R} $ in \cref{eq:Finite_DimensionalSDE} is full rank. Therefore, increasing finite dimensional approximation accuracy increases the complexity of the sampling process and optimal control computation. This is even more challenging in the case of SPDEs with $Q$-Wiener noise, where many of the eigenvalues in the expansion of $W(t)$ must be arbitrarily close to zero.

Other finite dimensional approaches as in \cite{Kappen2016} utilize Gaussian density functions instead of the measure theoretic approach. These approaches are not possible firstly due to the need to define the Gaussian density with respect to a measure other than the Lebesgue measure, which does not exist in infinite dimensions. Secondly, an equivalent Euler-Maruyama time-discretization is not possible without first discretizing spatially. Finally, after spatial discretization, the use of transition probabilities based on density functions requires invertibility of $\mathcal{R} \mathcal{R}^{\rT}$ (see Supplemental Material \cref{supsec:Comparison_finite}). These characteristics make Gaussian density based approaches not suitable for deriving optimal control of SPDEs.

%%%%%%%%%%%%%%%%%%%%SECTION CHANGE%%%%%%%%%%%%%%%%%%%
% \section{Numerical Results} \label{sec:Experiments}
%%%%%%%%%%%%%%%%%%%%%%%%%%%%%%%%%%%%%%%%%%%%%%%%%%%%%
\section{Numerical Results}

%%%%%%%%%%%%%%%%%%%%%%% Burgers%%%% %%%%%%%%%%%%%%%%%%%%%%%%%%%%%%%%%%%%%%%%%%%%%%%%%%%%%%%%%%%%%%%%%%%

%%%%%%%%%%%%FIGURE%%%%%%%%%%%
\begin{figure}[t!]
    \hspace{-0.175cm}
    {\includegraphics[width=\columnwidth]{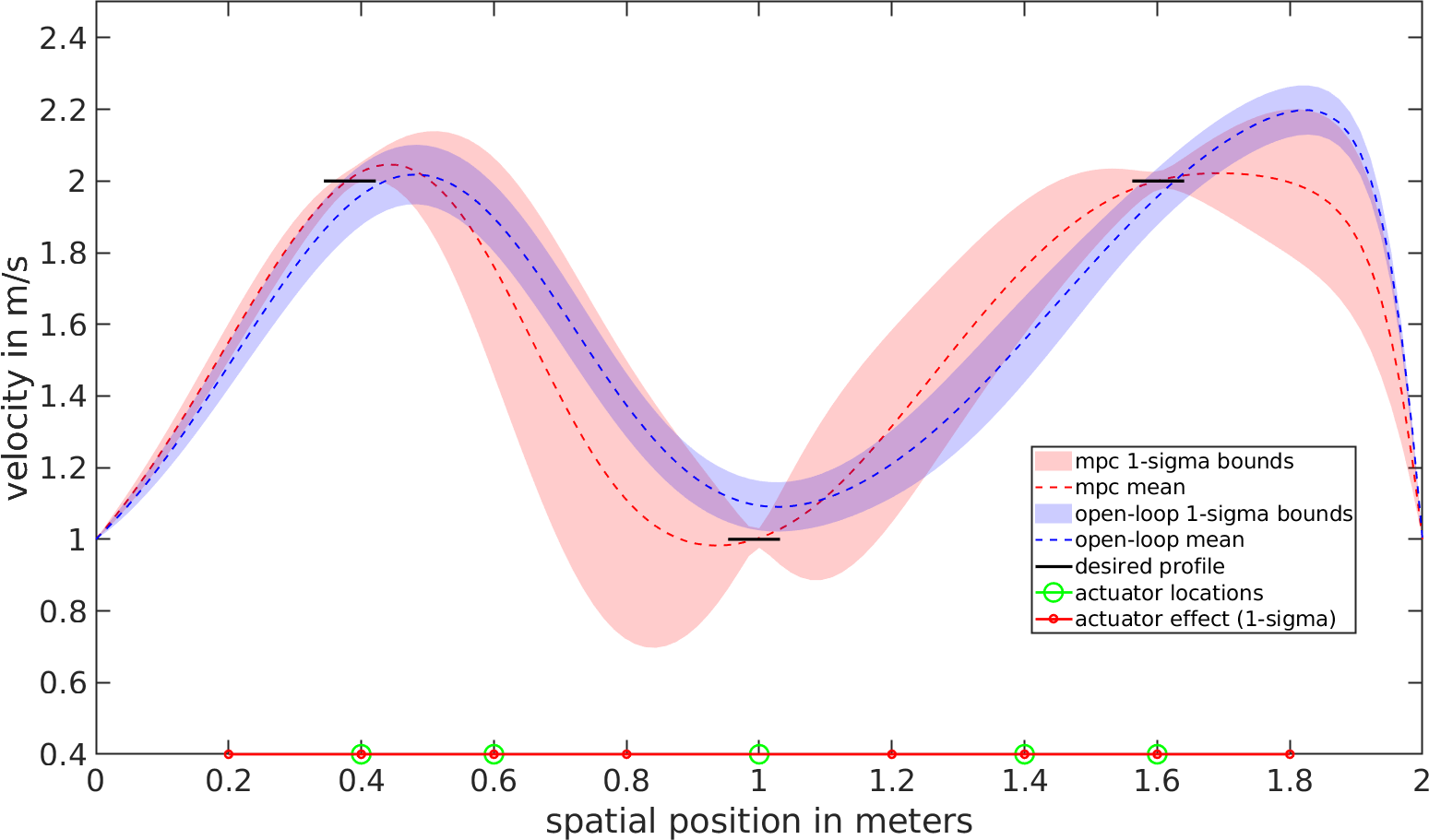}}
    
    {\includegraphics[width=0.49\columnwidth]{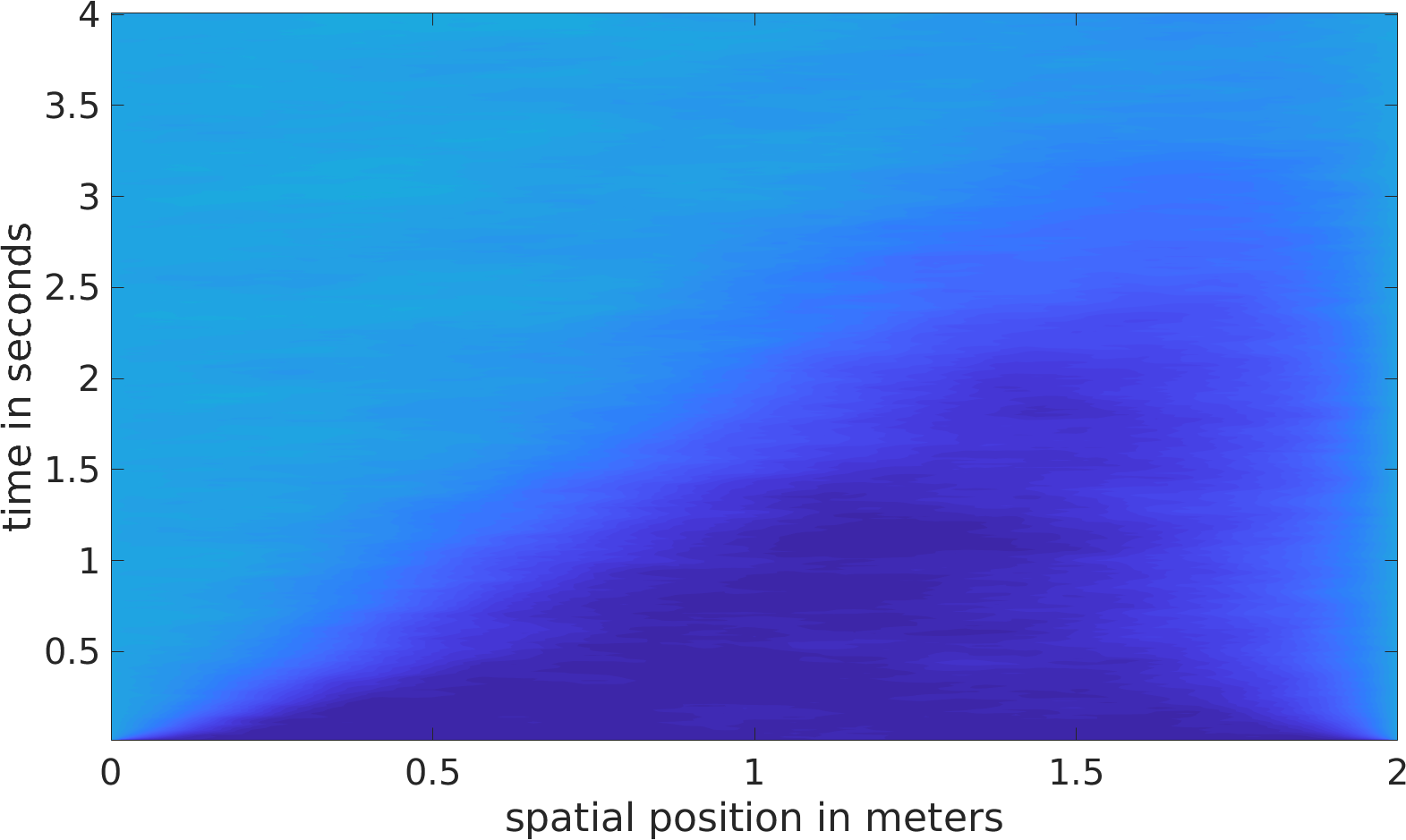}}
    \;\hspace{-0.2cm}
    % \caption{Uncontrolled dynamics}
    {\includegraphics[width=0.5\columnwidth]{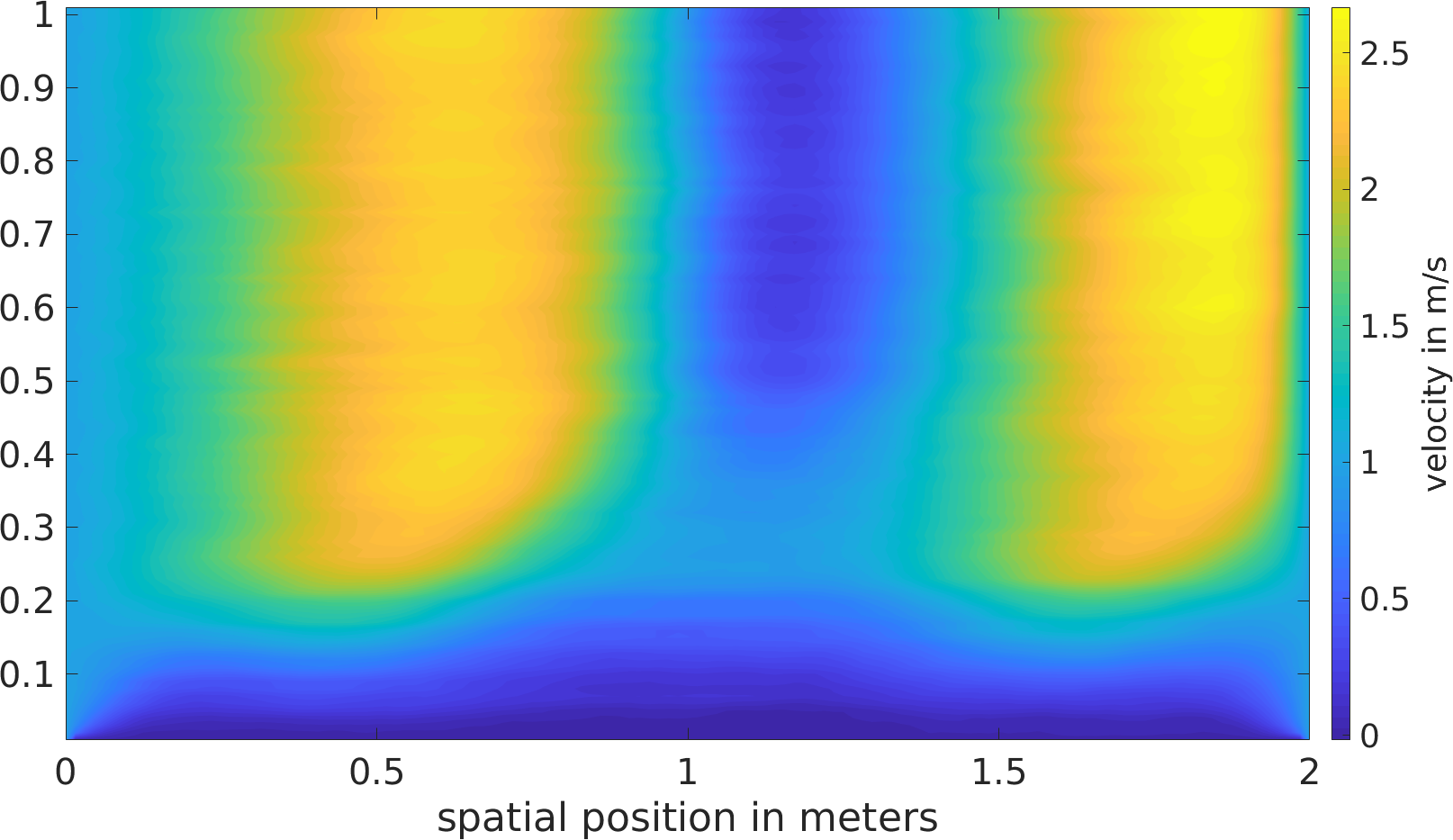}}
    % \caption{MPC to track desired profile}
\vspace{-0.5cm}
\caption{Infinite dimensional control of the 1-D Burgers SPDE: (top) Velocity profiles averaged over the \nth{2}-half of each time horizon over 128 trials. (bottom left) Spatio-temporal evolution of the uncontrolled 1-D Burgers SPDE with Cylindrical Wiener process noise. (bottom right) Spatio-temporal evolution of 1-D Burgers SPDE using MPC.}
\label{fig:Stochastic_Burgers}
\end{figure}
%%%%%%%%%%%%%%%%%%%%%%%%%%%%%

Performing variational optimization in the infinite dimensional space enables a general framework for controlling general classes of stochastic fields. It also comes with algorithmic benefits from importance sampling and can be applied in either open loop or MPC mode for both boundary and distributed control systems. Critically, it avoids feasibility issues in optimizing finite dimensional representations of SPDEs. Additional flexibility arises from the freedom to choose the model reduction method that is best suited for the problem without having to change the control update law. Details on the algorithm and more details on each simulated experiment can be found in the Supplemental Material \cref{supsec:results}.

\subsection{Distributed Control of Stochastic PDEs in Fluid Physics}

Several simulated experiments were conducted to investigate the efficacy of the proposed control approach. The first explores control of the 1-D stochastic viscous Burgers equation with non-homogeneous Dirichlet boundary conditions. This advection-diffusion equation with random forcing has been studied as a simple model for turbulence~\cite{da1994stochastic,jeng1969forced}.

The control objective in this experiment is to reach and maintain a desired velocity at specific locations along the spatial domain, depicted in black. In order to achieve the task, the controller must overcome the uncontrolled spatio-temporal evolution governed by an advective and diffusive nature, which produces an apparent velocity wave front that builds across the domain, as depicted on the bottom left of \cref{fig:Stochastic_Burgers}.

Both open-loop and MPC versions of the control in \cref{eq:Iterative_optimalvariation1} were tested on the 1-D stochastic Burgers equation and the results are depicted in the top subfigure of \cref{fig:Stochastic_Burgers}. Their performances are compared by averaging the voltage profiles for the \nth{2}-half of each experiment and repeated over 128 trials. The simulated experiment duration was 1.0 seconds. For the open-loop scheme, 100 optimization iterations with 100 sampled trajectory rollouts per iteration were used. In the MPC setting, 10 optimization iterations were performed at each time step, each using 100 sampled trajectory rollouts.

The results suggest that both the open-loop and MPC schemes have comparable success in controlling the Stochastic Burgers SPDE. The open-loop setting depicts the apparent rightward wavefront that is not as strong in the MPC setting. There is also quite a substantial difference in variance over the trajectory rollouts. The open-loop setting depicts a smaller variance overall, while the MPC setting depicts a variance that shrinks around the objective regions. The MPC performance is desirable since the performance metric only considers the objective regions. The Root Mean Squared Error (RMSE) and variance averaged over the desired regions is provided in \cref{table:Burgers}.

% \begin{table}[h!]
% \centering
%  \begin{tabular}{|| c | c | c ||} 
%  \hline
%   \textbf{Controller} & \textbf{MPC} & \textbf{Open-loop} \\ 
%  \hline
%  \textbf{RMSE} & \textcolor{red}{0.0231} & 0.0834 \\
%  \hline
%   \textbf{Avg. $\sigma$} & \textcolor{red}{0.0471} & 0.0780 \\
%   \hline
%  \end{tabular}
% % \caption{Summary of Monte Carlo trials for Burgers eqn}
% \label{table:Burgers}
% \end{table}

\begin{table}[H]
 \caption{\label{table:Burgers}Summary of Monte Carlo trials for the Stochastic Viscous Burgers Equation}
\begin{ruledtabular}
 \begin{tabular}{ c  c  c  c  c  c  c } 
  & \multicolumn{3}{c}{\textbf{RMSE}} & \multicolumn{3}{c}{\textbf{Average $\sigma$}} \\ 

 \textbf{Targets} & left & center & right & left & center & right \\   
 \hline
 \textbf{MPC} & \textcolor{red}{0.0344} & \textcolor{red}{0.0156} & \textcolor{red}{0.0132} & \textcolor{red}{0.0309} & 0.0718 & \textcolor{red}{0.0386} \\

  \textbf{Open-loop} & 0.0820 & 0.1006 & 0.0632 & 0.0846 & \textcolor{red}{0.0696} & 0.0797 \\

 \end{tabular}
 \end{ruledtabular}

\end{table}

%%%%%%%%%%%%%%%%%%%%%%%%%%%%%%%%%%%%%%%%%%%%%%%%%%%%%%%%%%%%%%%%%%%%%%%%%%%%%%%%%%%%%%%%%%%%%%%%%%%%

%%%%%%%%%%%%%%%%%%%%%%% Nagumo %%%% %%%%%%%%%%%%%%%%%%%%%%%%%%%%%%%%%%%%%%%%%%%%%%%%%%%%%%%%%%%%%%%%%%%

%%%%%%%%%%%%FIGURE%%%%%%%%%%%
\begin{figure*}[ht!] 

    {\hspace{-0.0cm}\includegraphics[width=0.492\textwidth]{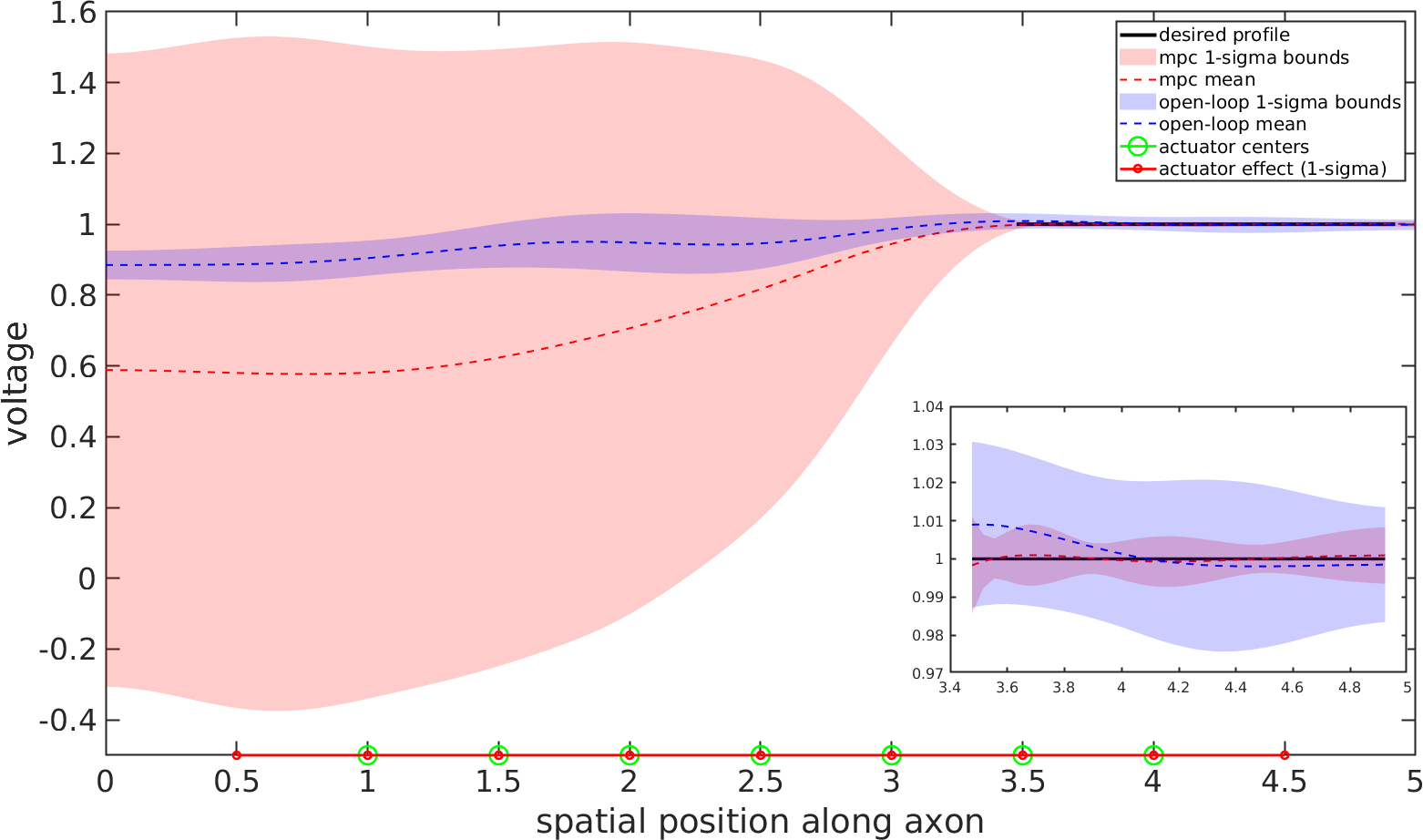}}
    % \small{Nagumo SPDE (Acceleration task)}
    \;\hspace{-0.0cm}
    {\includegraphics[width=0.492\textwidth]{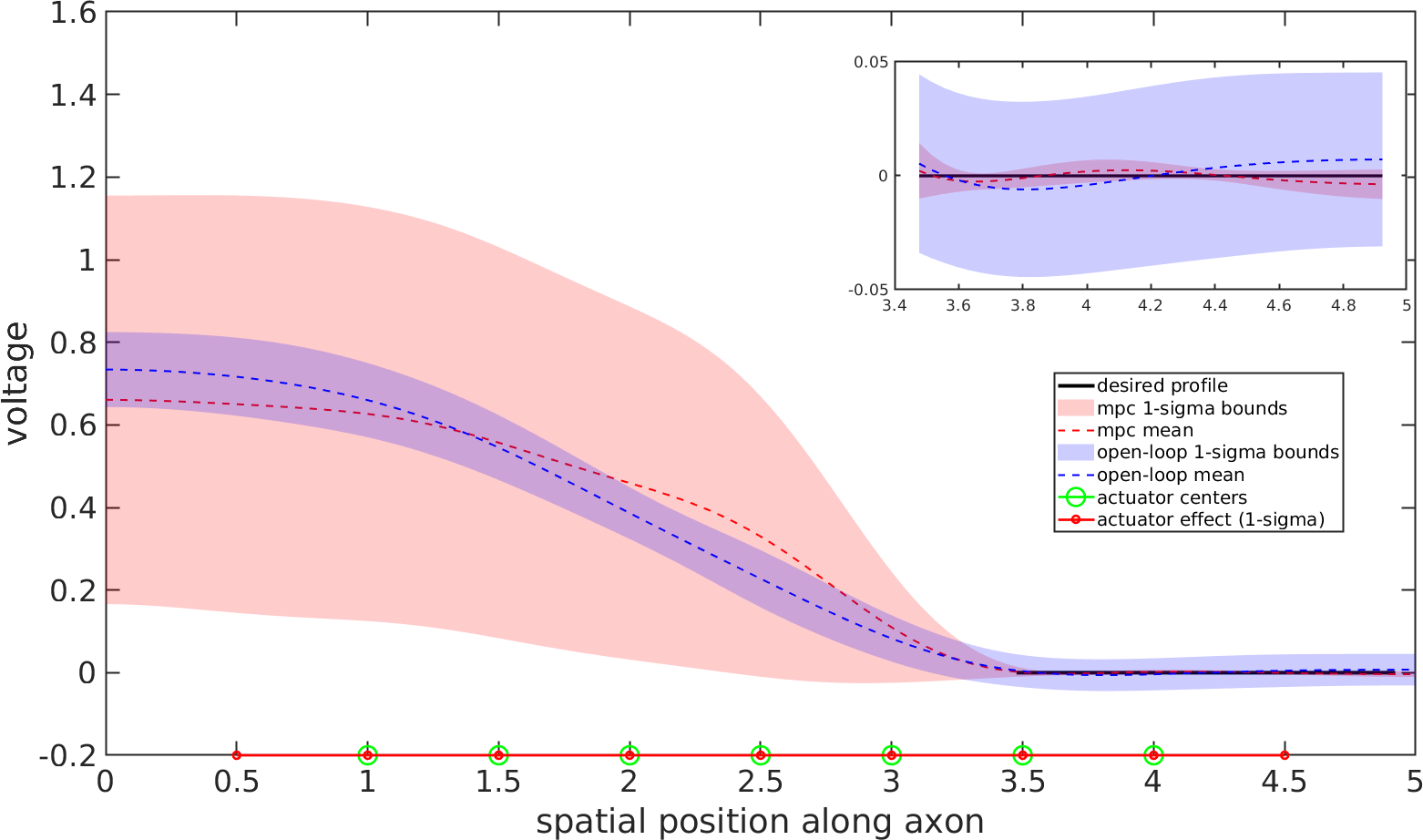}}
    % \small{Nagumo SPDE (Suppression task)}
    
    \hspace{-0.175cm}
    {\includegraphics[width=0.33\textwidth]{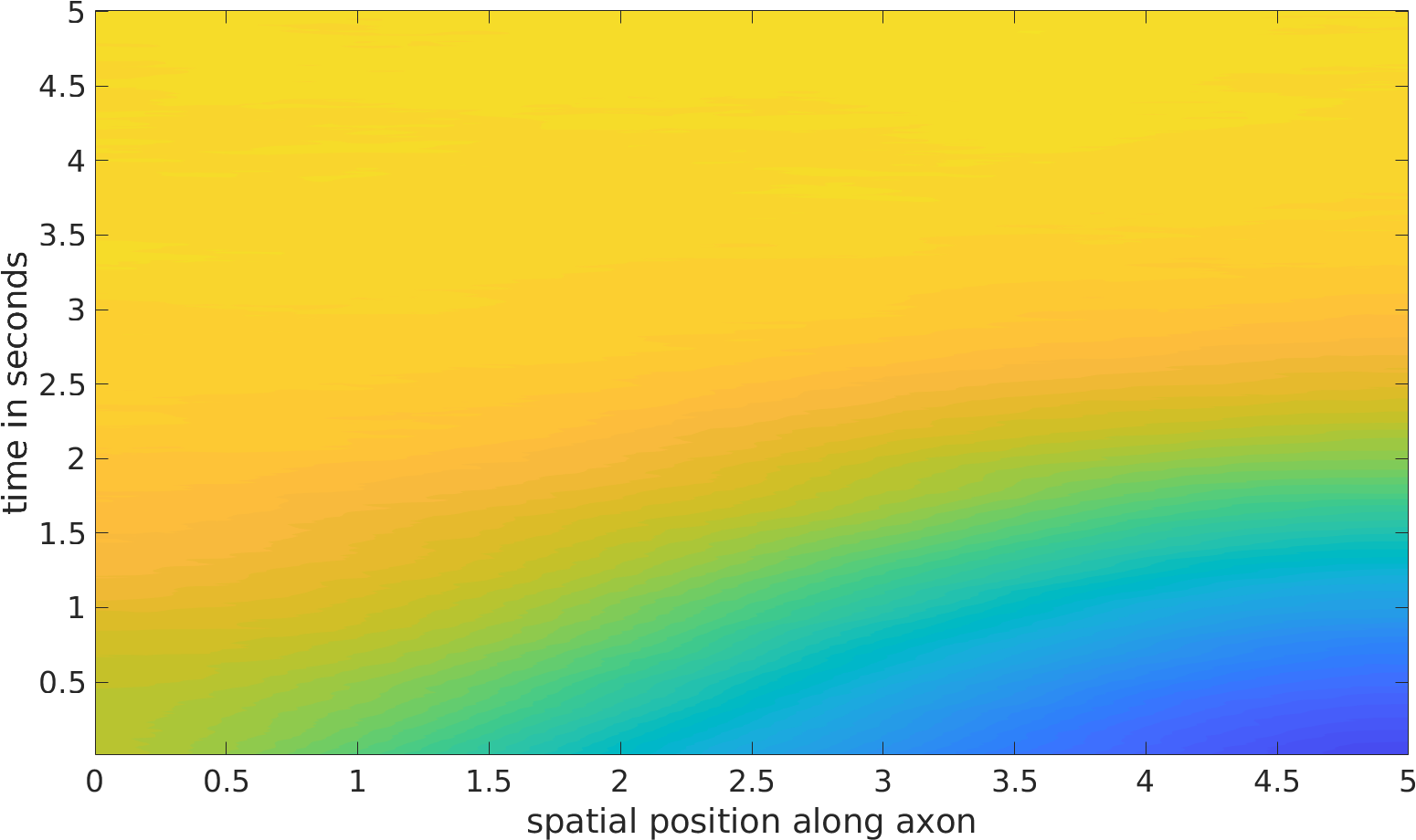}}
    % \small{(a) Uncontrolled Process}
    \;\hspace{-0.2cm}
        % \begingroup \captionof{subfigure}{} \endgroup
    % \;\hspace{-0.7cm}
    {\includegraphics[width=0.32\textwidth]{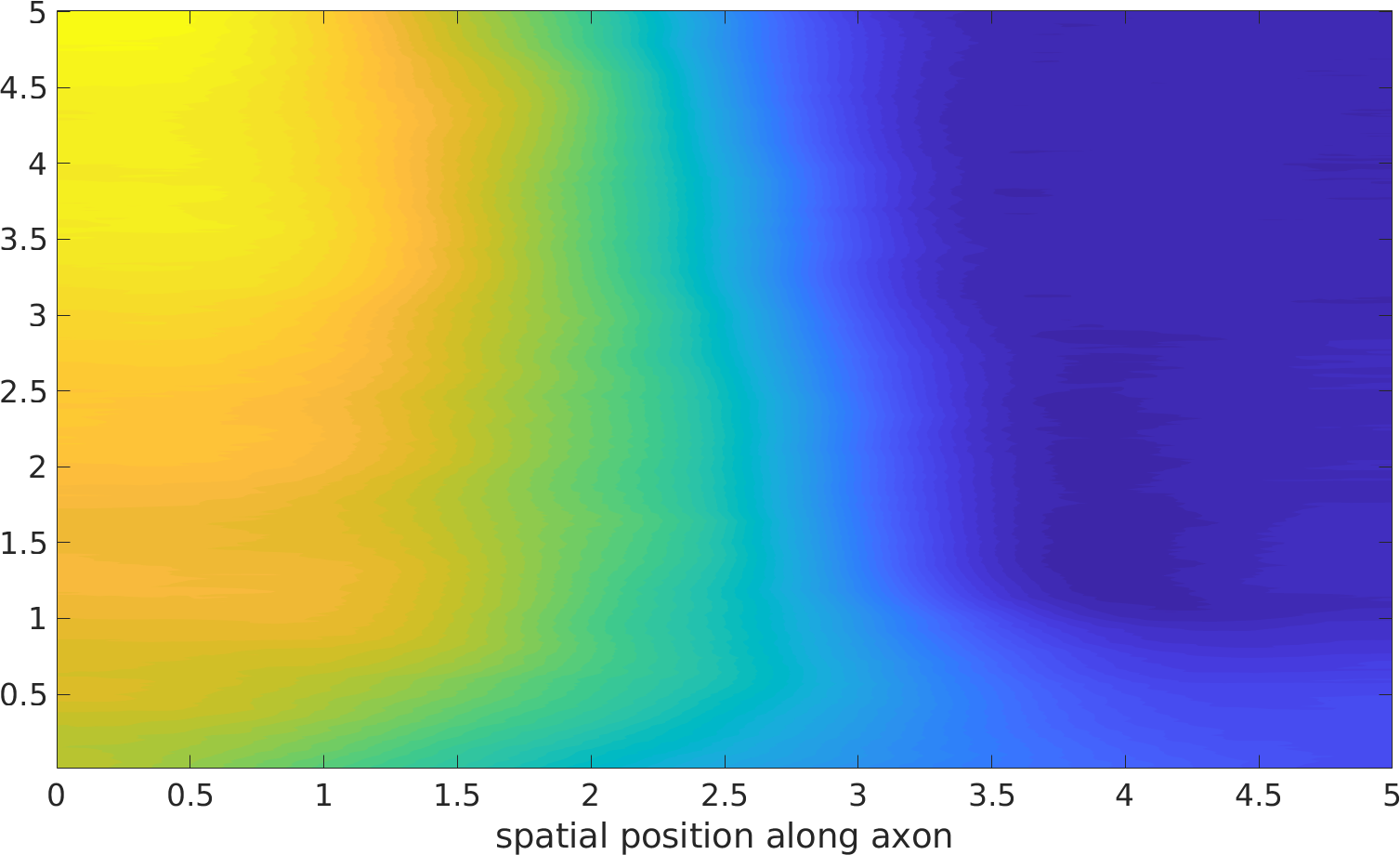}}
    % \small{(b) Suppression Task with MPC}
        %  \begingroup \captionof{subfigure}{b} \endgroup
    % \newline \centering
    \;\hspace{-0.2cm}
    {\includegraphics[width=0.335\textwidth]{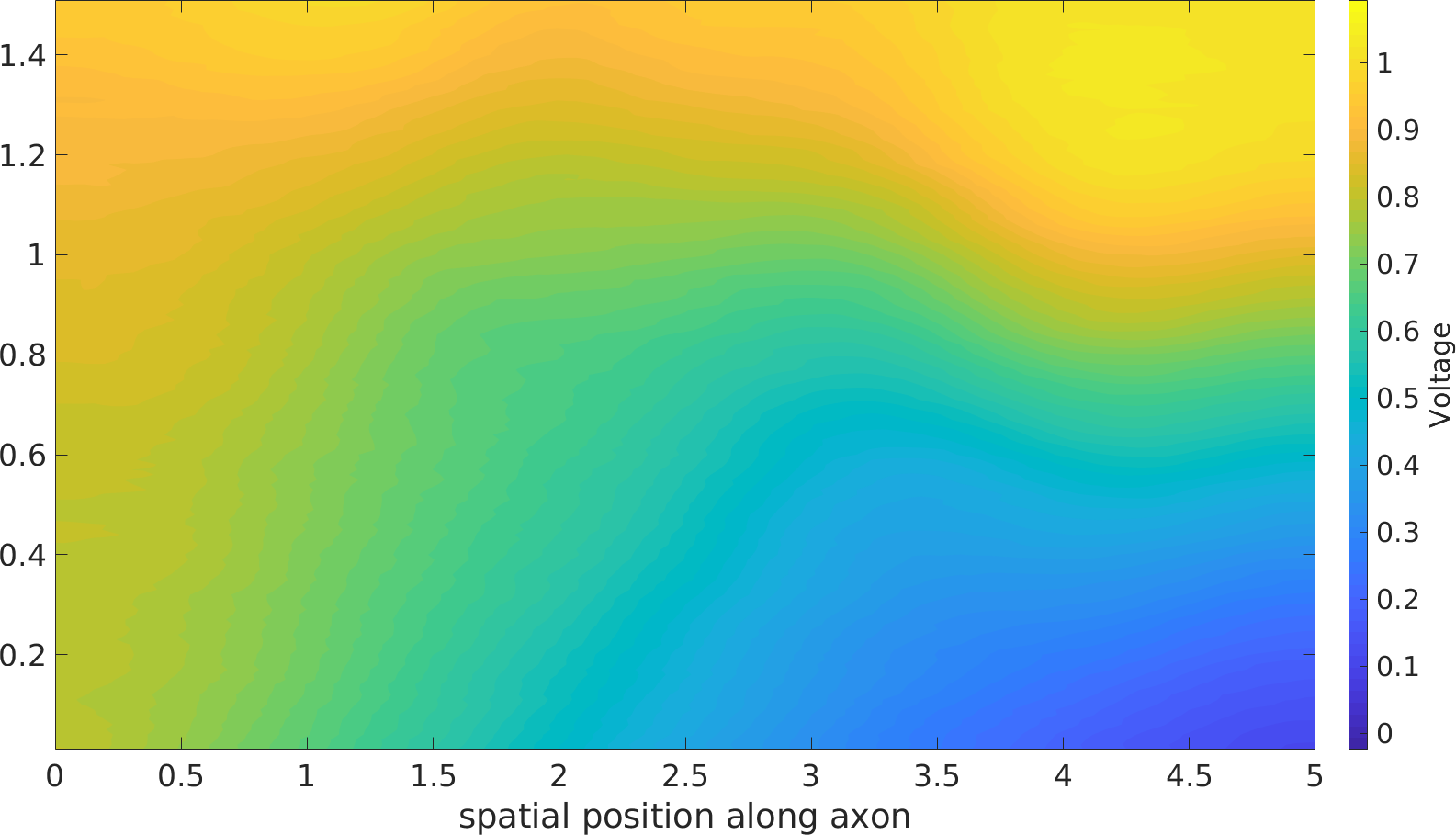}\vspace{0.25cm}}
    % \small{(c) Acceleration Task with MPC}
    
\caption{Infinite dimensional control of the Nagumo SPDE: (top left) and (top right) are voltage profiles averaged over the \nth{2}-half of each time horizon over 128 trials, (bottom left) uncontrolled spatio-temporal evolution for 5.0 seconds, (bottom center) suppressed activity with MPC for 5.0 seconds, and (bottom right) accelerated activity with MPC within 1.5 seconds.}
\label{fig:Stochastic_Nagumo}
\end{figure*}  
%%%%%%%%%%%%%%%%%%%%%%%%%%%%%

The stochastic Nagumo equation with homogeneous Neumann boundary conditions is a reduced model for wave propagation of the voltage in the axon of a neuron~\cite{lord_powell_shardlow_2014}. This SPDE shares a linear diffusion term with the Viscous Burgers equation, as depicted in \cref{tab:semilinear_pdes}. However, as shown in  the bottom left subfigure of \cref{fig:Stochastic_Nagumo}, the nonlinearity produces a substantially different behavior, which propagates the voltage across the axon with our simulation parameters in about 5 seconds. This set of simulated experiments explores two tasks: accelerating the rate at which the voltage propagates across the axon, and suppressing the voltage propagation across the axon. This is analogous to either ensuring the activation of a neuronal signal, or ensuring the neuron remains inactivated.

These tasks are accomplished by reaching either a desired value of 1.0 or 0.0 over the right end of the spatial region for acceleration and suppression, respectively. In both experiments, open-loop and MPC versions of \cref{eq:Iterative_optimalvariation1} were tested, and the results are depicted in \cref{fig:Stochastic_Nagumo}. For the open-loop scheme, 200 optimization iterations with 200 sampled trajectory rollouts per iteration were used. In the MPC setting, 10 optimization iterations were performed at each time step, each using 100 sampled trajectory rollouts. State trajectories of both control schemes were compared by averaging the voltage profiles for \nth{2}-half of each time horizon and repeated over 128 trials.

The results of the two stochastic Nagumo equation tasks suggest that both control schemes achieve success on both the acceleration and suppression tasks. While the performance appears substantially different outside the target region, the two control schemes have very similar performance on the desired region, which is the only penalized region in the optimization objective. In the top subfigures of \cref{fig:Stochastic_Nagumo}, the desired region is zoomed in on. The zoomed in views depict a higher variance in the state trajectories of the open-loop control scheme than the MPC scheme. 

As in the stochastic viscous Burgers experiment, there is an apparent trade-off between the two control schemes. The MPC scheme yields a desirable lower variance in the region that is being considered for optimization, but produces state trajectories with very high variance outside the goal region. The open loop control is understood as seeking to achieve the task by reaching low variance trajectories everywhere, while the MPC scheme is understood as acting reactively (i.e. re-optimizes based on state measurements) to a propagating voltage signal. The RMSE and variance averaged over the desired region of 128 trials of each experiment is given in \cref{table:Nagumo}.

\begin{table}[H]
\captionof{table}{\label{table:Nagumo} Summary of Monte Carlo trials for Nagumo acceleration and suppression tasks}
\begin{ruledtabular}
 \begin{tabular}{ c  c  c  c  c } 
  \textbf{Task} & \multicolumn{2}{c}{\textbf{Acceleration}} & \multicolumn{2}{c}{\textbf{Suppression}} \\ 
\textbf{Paradigm} & MPC & Open-Loop & MPC & Open-Loop \\
 \hline
 \textbf{RMSE} & \textcolor{red}{6.605$e^{-4}$} & 0.0042 & \textcolor{red}{0.0021} & 0.0048 \\
 
  \textbf{Avg. $\sigma$} & \textcolor{red}{0.0059} & 0.0197  & \textcolor{red}{0.0046}  & 0.0389 \\
 \end{tabular}
\end{ruledtabular}
% \caption{Summary of Monte Carlo trials for Nagumo tasks}
\end{table}

%%%%%%%%%%%%%%%%%%%%%%% Heat %%%%%%%%%%%%%%%%%%%%%%%%%%%%%%%%%%%%%%%%%%%%%%%%%%%%%%%%%%%%%%%%%%%%%%%

%%%%%%%%%%%%FIGURE%%%%%%%%%%%
\begin{figure*}[t!]
    \hspace{-0.35cm}
        \includegraphics[width=0.25\textwidth]{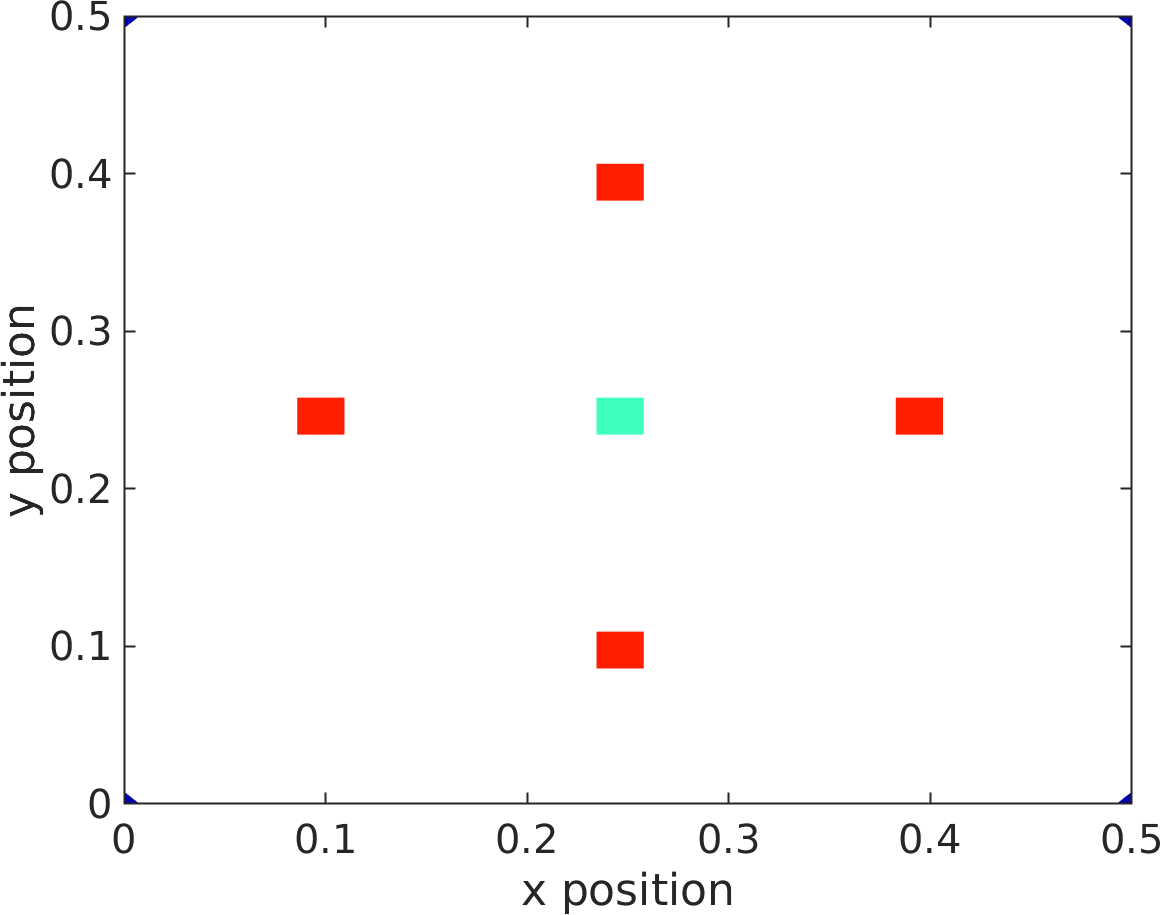}
        % \caption{Desired Profile}
        % \label{fig:heat_desired}
    \;\hspace{-0.15cm}
        \includegraphics[width=0.2395\textwidth]{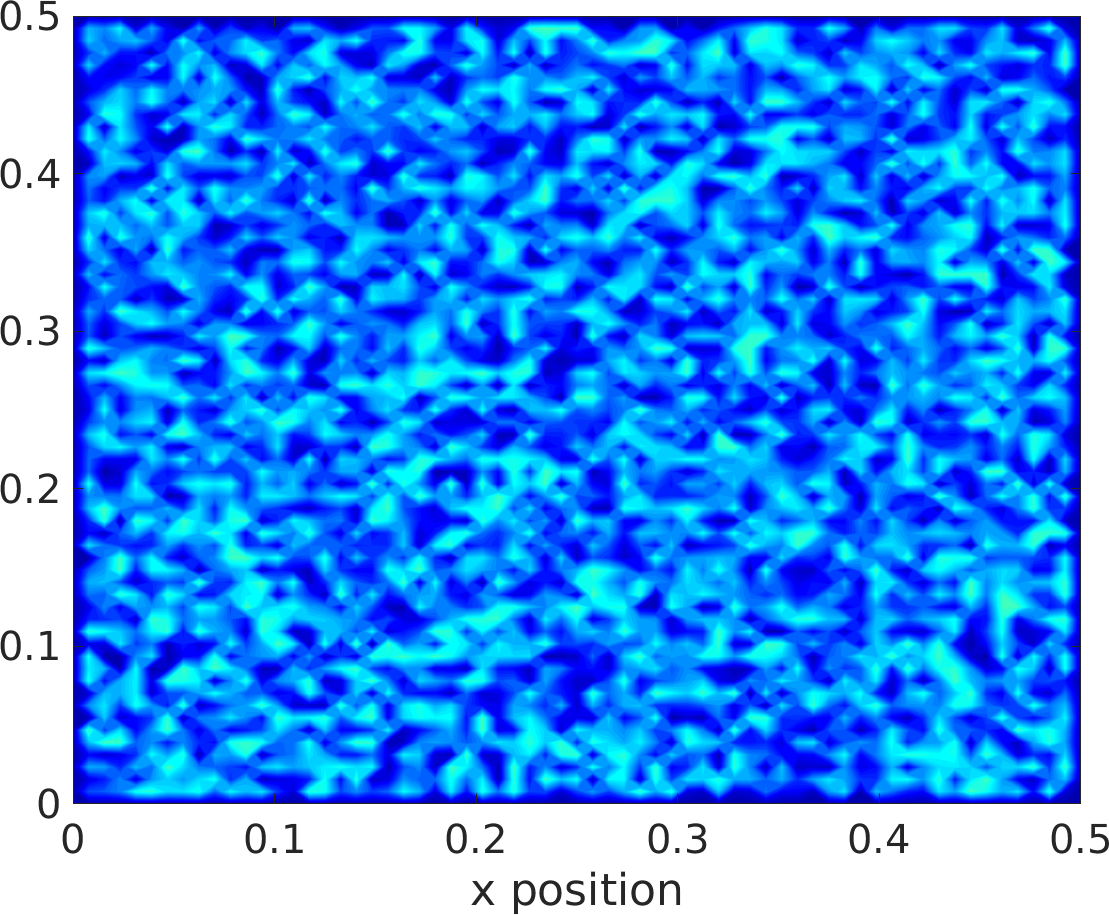}
        % \caption{Random Initial Profile}
        % \label{fig:heat_start}
    \;\hspace{-0.15cm}
        \includegraphics[width=0.224\textwidth]{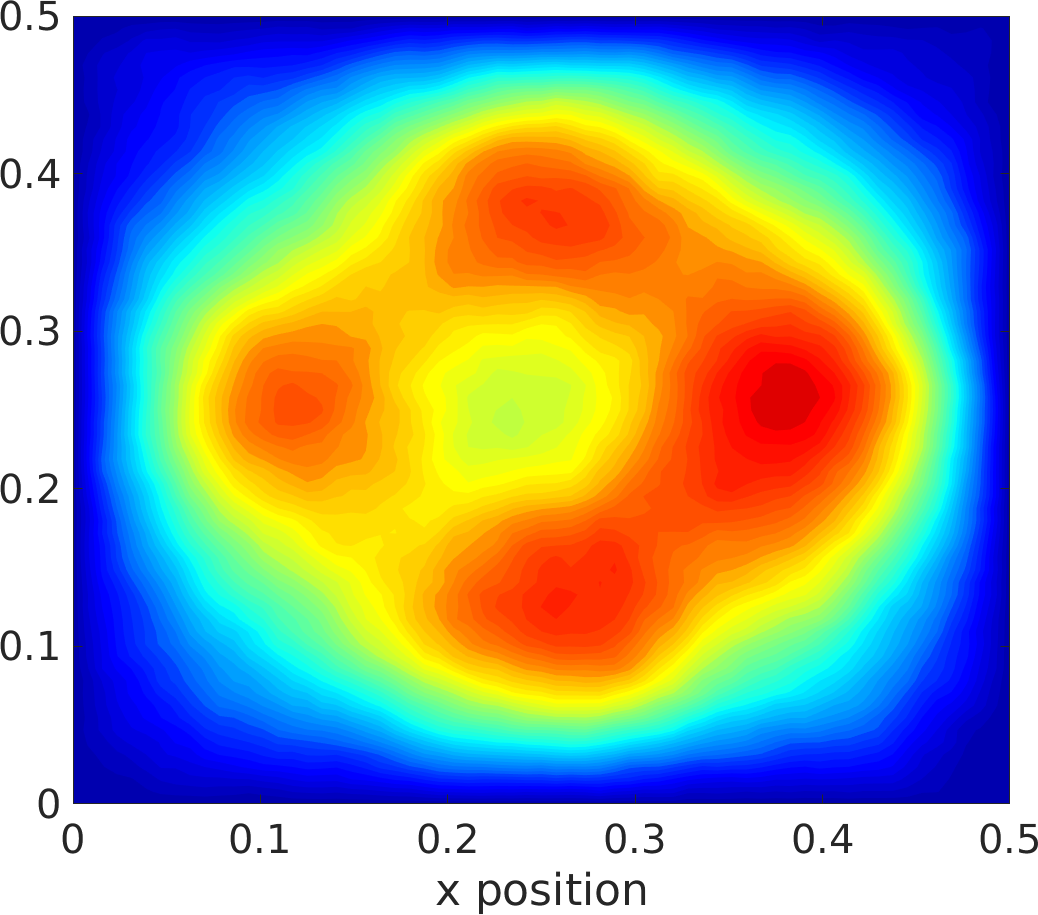}
        % \caption{Profile Half-way}
        % \label{fig:heat_halway}
    \;\hspace{-0.15cm}
        \includegraphics[width=0.2613\textwidth]{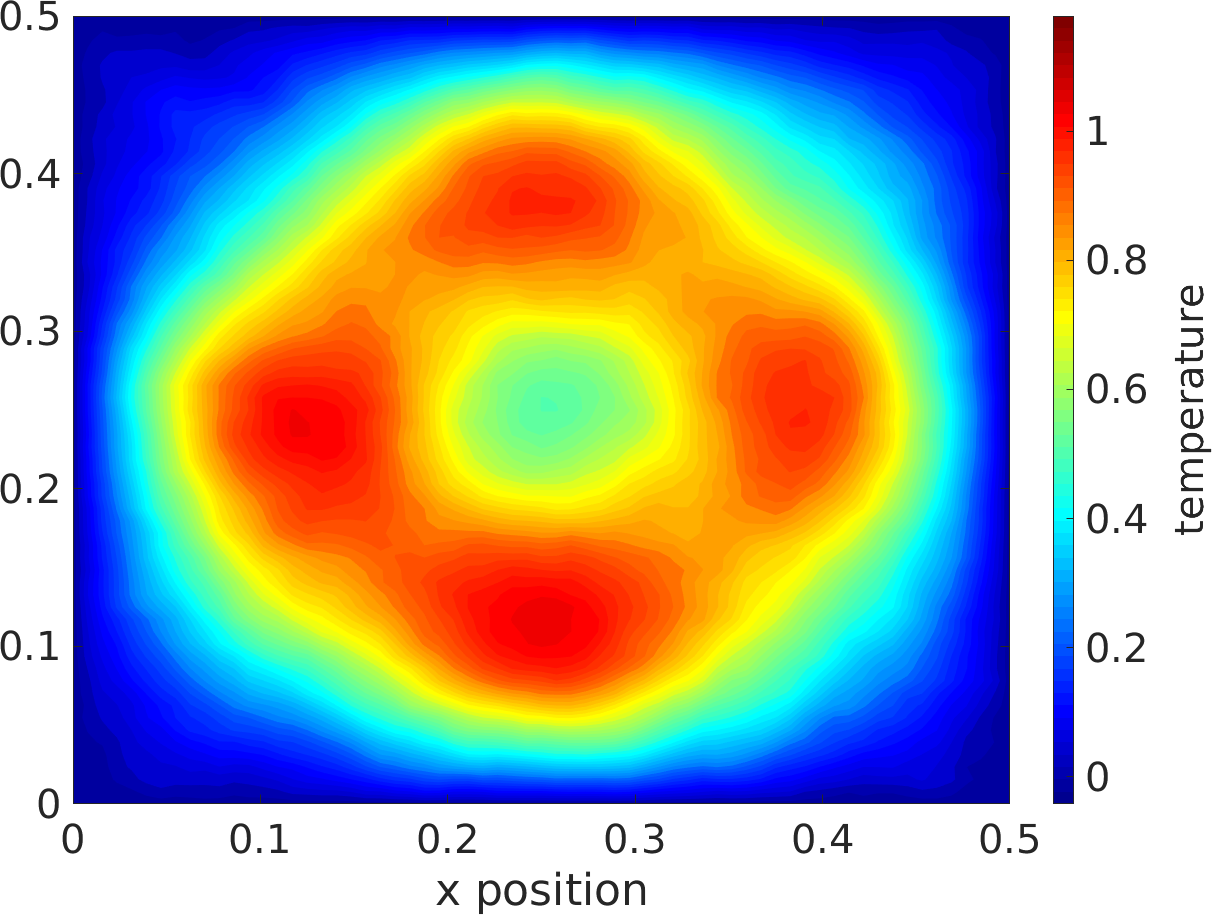}
        % \caption{Profile at the End}
        % \label{fig:heat_end}
    
    \caption{Infinite Dimensional control of the 2D-Heat SPDE under homogeneous Dirichlet boundary conditions: (first) desired temperature values at specified spatial regions, (second) random initial temperature profile, (third) temperature profile half way through the experiment and (fourth) temperature profile at the end of experiment.}
    \label{fig:Stochastic_Heat}
\end{figure*}
%%%%%%%%%%%%%%%%%%%%%%%%%%%%%

\begin{figure*}[ht!]
    \hspace{-.7cm}
    {\includegraphics[width=0.47\textwidth]{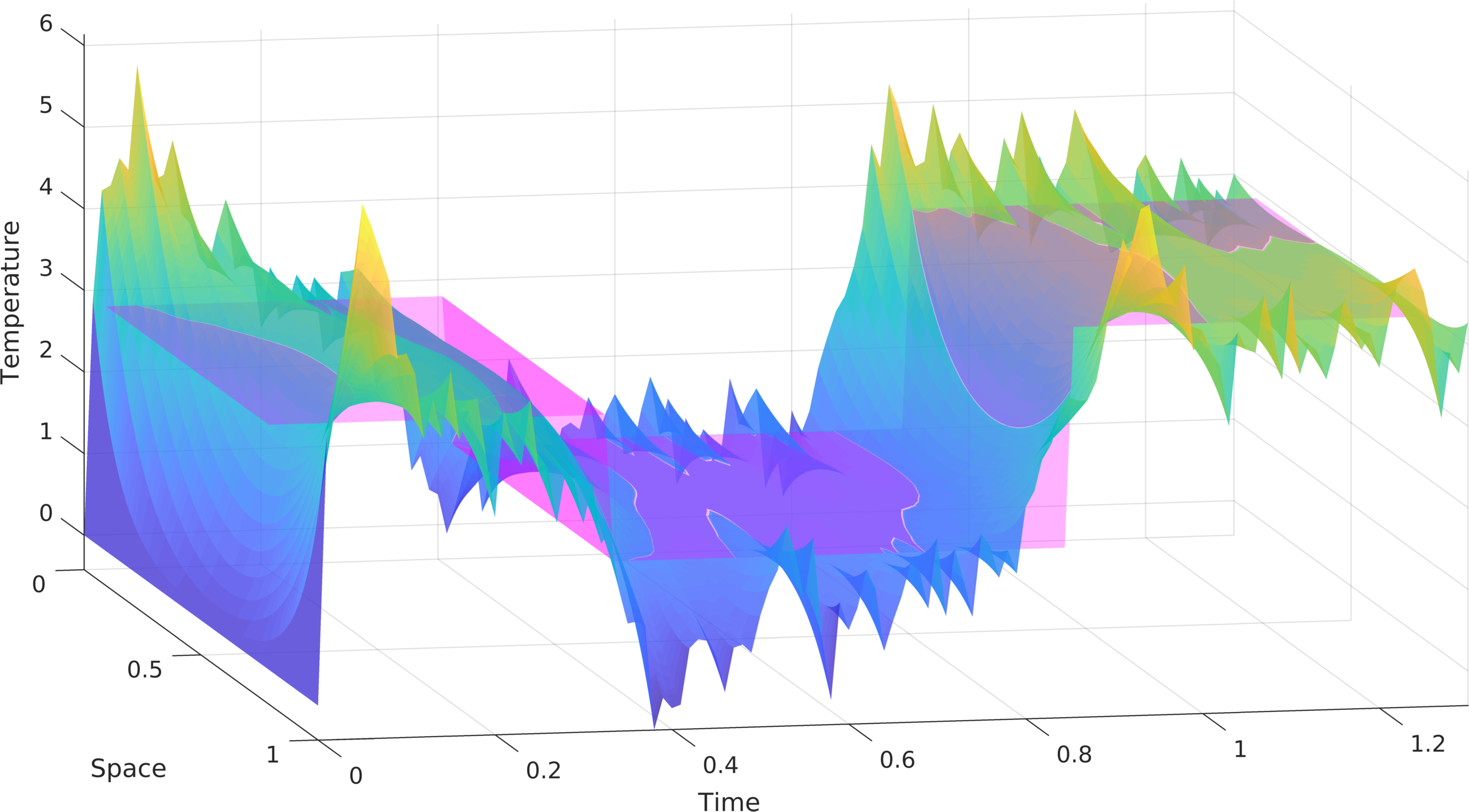}}
    %  \caption{Temperature Profile}
    %  \label{fig:bound_h}
	\;\hspace{-.01cm}
    {\includegraphics[width=0.47\textwidth]{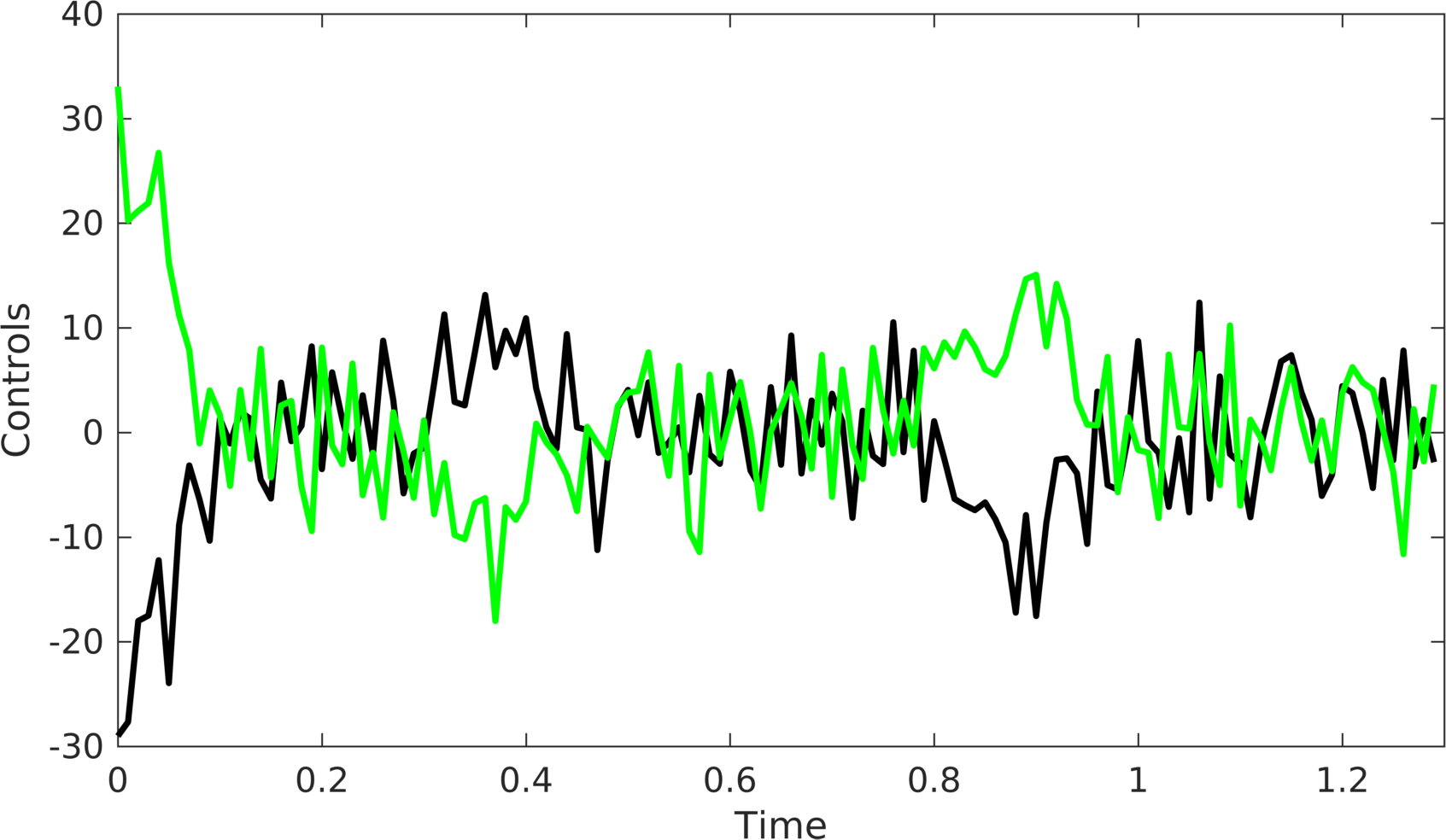}}
    %  \caption{Control Inputs}
    %  \label{fig:bound_u}
	\caption{Boundary control of stochastic 1-D heat equation: (left) Temperature profile over the 1D spatial domain over time. The magenta surface corresponds to the spatio-temporal desired temperature profile. Colors that are more red correspond to higher temperatures, and colors that are more violet correspond to lower temperature. (right) Control inputs at the left boundary in black and the right boundary in green entering through Neumann boundary conditions.}
	\label{fig:Stochastic_Boundary_Heat}
\end{figure*}

The next simulated experiment explores scalability to 2D spatial domains by considering the 2D stochastic heat equation with homogeneous Dirichlet boundary conditions. This experiment can be thought of as attempting to heat an insulated metal plate to specified temperatures in specified regions while the edges remain at a temperature of 0 in some scale. The desired temperatures and regions associated with this experiment are depicted in the left subfigure of \cref{fig:Stochastic_Heat}. This experiment tests the MPC scheme. 

Starting from a random initial temperature profile as in the second subfigure of \cref{fig:Stochastic_Heat}, and using a time horizon of 1.0 seconds, the MPC controller is able to achieve the desired temperature profile towards the end of the time horizon as shown in the fourth subfigure of \cref{fig:Stochastic_Heat}. The third subfigure of \cref{fig:Stochastic_Heat} depicts the middle of the time horizon. The MPC controller used 5 optimization iterations at every timestep and 25 sampled trajectories per iteration.

This result suggests that in this case, this approach can handle the added complexity of 2D stochastic fields. As depicted in the right subfigure of \cref{fig:Stochastic_Heat}, the proposed MPC control scheme solves the task of reaching the desired temperature at the specified spatial regions.

%%%%%%%%%%%%%%%%%%%%%%%%%%%%%%%%%%%%%%%%%%%%%%%%%%%%%%%%%%%%%%%%%
% \subsection{Boundary Control of Stochastic Parabolic Equations} \label{sec:bound}
%%%%%%%%%%%%%%%%%%%%%%%%%%%%%%%%%%%%%%%%%%%%%%%%%%%%%%%%%%%%%%%%%
\subsection{Boundary Control of Stochastic PDEs}\label{sec:bound}

The control update in \cref{eq:Iterative_optimalvariation1} describes control of SPDEs by distributing actuators throughout the field. However, our framework can also handle systems with control and noise at the boundary. A key requirement is to write such dynamical systems in the {\it mild} form
\begin{equation} \label{eq:mild_bound}
\small\begin{split}
&X(t)=e^{t\mathcal{A}}\xi+\int_{0}^{t}e^{(t - s)\mathcal{A}}F_1(t, X)\mathrm ds\\
&+(\lambda I - \mathcal{A})\bigg[\int_{0}^{t}e^{(t - s)\mathcal{A}}\mathsf{D}\big(F_2(t, X)+G(t, X)\calU(t, X; \vtheta)\big)\mathrm ds\\
&\hspace{18mm}+\int_{0}^{t}e^{(t - s)\mathcal{A}}\mathsf{D}B(t, X)\mathrm dV(s)\bigg], \quad \mathbb{P}\hspace{1mm} a.s.
\end{split}
\end{equation}\normalsize
where the operator $\mathsf{D}$ corresponds to the boundary conditions of the problem, and is called the {\it Dirichlet map} ({\it Neumann map}, resp.) for Dirichlet (Neumann, resp.) boundary control/noise. These maps take operators defined on the boundary Hilbert space $\Lambda_0$ to the Hilbert space $H$ of the domain. $\lambda$ is a real number also associated with the boundary conditions. The operator $\rd V$ describes a cylindrical Wiener process on the boundary Hilbert space $\Lambda_0$. For further details, the reader can refer to the discussion in \cite[Section 2.5 \& Appendix C.5]{fabbri} and in the Supplemental Material \cref{supsec:BoundarySPDEs}.

Studying optimal control problems with dynamics as in \cref{eq:mild_bound} is rather challenging. Therein HJB theory requires additional regularity conditions and proving convergence of \cref{eq:mild_bound} becomes nontrivial, especially when considering Dirichlet boundary noise. Numerical results are limited to simplistic problems. Nevertheless, \cref{eq:Iterative_optimalvariation1} is extended to the case of boundary control by similarly using tools from Girsanov's theorem to obtain the change of measure
\begin{equation}
\begin{split}\label{eq:boundary_girsanov}
     \frac{\rd \calL}{\rd \calL^{(i)} }=\exp \bigg(&-\int_{0}^{T}\big\langle B^{-1}G\calU, \rd V(s)\big\rangle_{\Lambda_{0}}\\ 
     &+\frac{1}{2}\int_{0}^{T}||B^{-1}G\calU||_{\Lambda_{0}}^{2} \rd s\bigg),
\end{split}
\end{equation}
which was was also utilized in reference \cite{duncan1998ergodic} for studying solutions of SPDEs similar to \cref{eq:mild_bound}. Using the control parameterization of the distributed case above results in the same approach described in \cref{eq:Iterative_optimalvariation1} with inner products taken with respect to the boundary Hilbert space $\Lambda_0$ to solve stochastic boundary control problems.

The stochastic 1-D heat equation under Neumann boundary conditions was explored to conduct simulated experiments that investigate the efficacy of the proposed framework in stochastic boundary control settings. The objective is to track a time-varying profile that is uniform in space by actuation only at the boundary points. The MPC scheme of \cref{eq:Iterative_optimalvariation1}, with 10 optimization iterations per time step is depicted in the left subfigure of \cref{fig:Stochastic_Boundary_Heat}. The random sample of the controlled state trajectory, depicted in a violet to red color spectrum, remains close to the time-varying desired profile, depicted in magenta. The associated bounded actuation signals acting on the two boundary actuators are depicted in the right subfigure of \cref{fig:Stochastic_Boundary_Heat}.

The above experiments were designed to cover stochastic SPDEs with nonlinear dynamics, multiple spatial dimensions, time-varying objectives, and systems with both distributed and boundary actuation. This range explores the versatility of the proposed framework to problems of many different types. Throughout these experiments, the control architecture produces state trajectories that solve the objective with high probability for the given stochasticity.

\section{Conclusion}

%%%%%%%%%%%%%%%%%%%%%%%SECTION CHANGE%%%%%%%%%%%%%%%%%%%%%%%%%%%%
% \section{Conclusion} \label{sec:conclusion}
%%%%%%%%%%%%%%%%%%%%%%%%%%%%%%%%%%%%%%%%%%%%%%%%%%%%%%%%%%%%%%%%%

This manuscript presented a variational optimization framework for distributed and boundary control of stochastic fields based on the free energy-relative entropy relation. The approach leverages the inherent stochasticity in the dynamics for control, and is valid for generic classes of infinite-dimensional diffusion processes. Based on thermodynamic notions that have demonstrated connections to established stochastic optimal control principles, algorithms were developed that bridge the gap between abstract theory and computational control of SPDEs. The distributed and boundary control experiments demonstrate that this approach can successfully control complex physical systems in a variety of domains.

This research opens new research directions in the area of control of stochastic fields that are ubiquitous in domains of physics. Based on the use of forward sampling, future research on the algorithmic side will include the development of efficient methods for representation and propagation of stochastic fields using techniques in machine learning such as Deep Neural Networks. Other directions include explicit feedback parameterizations and, in the context of boundary control, HJB approaches in the information theoretic formulation.

\section{Acknowledgments}
This work was supported by the Army Research Office contract W911NF2010151. Ethan N. Evans was supported by the SMART scholarship and George I. Boutselis was partially supported by  the Onassis Fellowship.

\bibliography{References}% Produces the bibliography via BibTeX.

\clearpage
\newpage
\onecolumngrid
\beginsupplement
\section*{Supplemental Material}

%%%%%%%%%%%%%%%%%%%%%%%%%%%%%%%%%%%%%%%%%%%%%%%%%%%%%%%%%%%%%%%%%%%%%%%%%%%%%%%%%%%%%%%%%%%%%%%%%%NEW_SECTION%%%%%%%%%%%%%%%%%%%%%%%%%%%%%%%%%%%%%%%%%%%%%%%%%%%%%%%%%%%%%%%%%%%%%%%%%%%%%%%%%%%%%%%%%%%%%%%%%%%%%%%%%%%
\section{Theoretical Results: Variational Optimization, Stochastic Optimal Control, and Thermodynamic Inequalities}

\subsection{Description of the Hilbert Space Wiener Process}\label{supsec:wiener}

In this section we provide formal definitions of various forms of the Hilbert space Wiener process. Some of these statements can be found in \cite[Section 4.1]{da2008stochastic}.

\begin{definition} \label{def:wiener_process}
Let $\calH$ denote a Hilbert space. A $\calH$-valued stochastic process $W(t)$ with probability law $\calL\big(W(\cdot)\big)$ is called a Wiener process if
\begin{enumerate}
    \item $W(0) = 0$
    \item $W$ has continuous trajectories
    \item $W$ has independent increments
    \item $\calL\big(W(t) - W(s)\big) = \calN \big(0, (t-s)Q\big), \quad t \geq s \geq 0$
    \item $\calL\big(W(t)\big) = \calL\big(-W(t)\big), \quad t \geq 0$
\end{enumerate}
\end{definition}

\begin{proposition}\label{prop:expansion}
Let $ \{e_{i}\}_{i=1}^{\infty} $  be a complete orthonormal system for the Hilbert Space $\calH$. Let $Q$ denote the covariance operator of the Wiener process $W(t)$. Note that $Q$ satisfies $Q e_{i}= \lambda_{i} e_{i}$, where $\lambda_{i}$ is the eigenvalue of $Q$ that corresponds to eigenvector $e_{i}$. Then, $W(t) \in \calH$ has the following expansion:
 \begin{equation}\label{eq:Wiener_expansion}
    W(t) = \sum_{j=1}^{\infty} \sqrt{\lambda_{j}} \beta_{j}(t) e_{j},
\end{equation}
\noindent where  $ \beta_{j}(t)  $  are real valued Brownian motions that are mutually independent on $ (\Omega, \calF, \mathbb{P})$.
\end{proposition}

\begin{definition}\label{def:trace_class}
Let $\{e_{i}\}_{i=1}^{\infty}$ be a complete orthonormal system for the Hilbert Space $\calH$. An operator $A$ on $\calH$ with the set of its eigenvalues $\{\lambda_i\}_{i=1}^{\infty}$ in a given basis $\{e_i\}_{i=1}^{\infty}$ is called a \textit{trace-class} operator if
\begin{equation}\label{eq:trace_class}
   \mbox{Tr}(A) := \sum_{n=1}^\infty \big\langle A e_n, e_n \big \rangle = \sum_{i=1}^{\infty} \lambda_i < \infty.
\end{equation}
\end{definition}

The two primary Wiener processes that are typically used to model spatio-temporal noise processes in the SPDE literature are the Cylindrical Wiener process and the $Q$-Wiener process. These are both referred to in the main text, and are defined in the following two definitions.

\begin{definition}\label{def:cylindrical_wiener}
A Wiener process $W(t)$ on $\calH$ with covariance operator $Q$ is called a \textit{Cylindrical Wiener process} if $Q$ is the identity operator $I$.
\end{definition}

\begin{definition}\label{def:q_wiener}
A Wiener process $W(t)$ on $\calH$ with covariance operator $Q$ is called a \textit{$Q$-Wiener process} if $Q$ is of trace-class.
\end{definition}

% The following facts follow immediately from \cref{def:cylindrical_wiener,def:q_wiener}.

% \begin{proposition}
% Let $W(t)$ be a Cylindrical Wiener process on $\calH$.
% \end{proposition}

An immediate fact following \cref{def:cylindrical_wiener} is that the Cylindrical Wiener process acts spatially \textit{everywhere} on $\calH$ with equal magnitude. One can easily conclude that for a Cylindrial Wiener process, the eigenvalues $\{\lambda_i\}_{i=1}^{\infty}$ of the covariance operator $Q$ are all unity, thus
\begin{equation}
    \sum_{i=1}^{\infty} \lambda_i = \infty.
\end{equation}
However, we note that in this case the series in \eqref{eq:Wiener_expansion} converges in another Hilbert space $U_{1}\supset U$, when the inclusion $\iota:U\rightarrow U_{1}$ is Hilbert-Schmidt. For more details see \cite{da1992stochastic}. 

On the other hand, immediately following \cref{def:q_wiener} is the fact that 
% the $Q$-Wiener process has a trace-class covariance operator satisfying \cref{eq:trace_class}, and thus 
a $Q$-Wiener process
must not have a spatially equal effect everywhere on the domain. More precisely, one has the following proposition.
\begin{proposition}
Let $W(t)$ be a $Q$-Wiener process on $\calH$ with covariance operator $Q$. Let $\{\lambda_i\}_{i=1}^{\infty}$ denote the set of eigenvalues of $Q$ in the complete orthonormal system $\{e_i\}_{i=1}^{\infty}$. Then the eigenvalues must fall into one of the following three cases:\\
i) For any $\varepsilon > 0$, there are only finitely many eigenvalues $\lambda_i$ of covariance operator $Q$ such that $|\lambda_i| > \varepsilon$. That is, the set $\{i \in \Nb_+ : \lambda_i > \varepsilon \}$, where $\Nb_+$ is the positive natural numbers, has finite elements. \\
ii) The eigenvalues $\lambda_i$ of covariance operator $Q$ follow a bounded periodic function such that $|\lambda_i|>0$ $\forall$ $i \in \Nb_+$ and $\sum_{i=1}^{\infty} \lambda_i = 0$.\\
iii) Both case i) and case ii) are satisfied. In this case the eigenvalues follow a bounded and convergent periodic function with $\lim_{i\rightarrow\infty} \lambda_i = 0$.
\end{proposition}
% More precisely, the $Q$-Wiener process must fall into one of two cases: \\

% Additionally, for $Q$-Wiener processes one has the following:

% \begin{prop} Let $ \{e_{i}\}_{i=1}^{\infty} $  be a complete orthonormal system for the Hilbert Space $U$  such that $ Q e_{i}= \lambda_{i} e_{i}   $. Here, $ \lambda_{i} $ is the eigenvalue of $ Q\in L(U) $ that corresponds to  eigenvector $ e_{i} $.  Then,  a   $Q$-Wiener process $ W(t) \in U$ satisfies the following properties
%  \begin{enumerate}
%  \item  $ W $ is a Gaussian process on $ U $  with mean and variance: 
%  \begin{equation}
%  \Eb [W(t)] = 0, \quad   \Eb [W(t) W(t) ] = t Q, ~ t\geq 0. 
%  \end{equation} 
 
%  \item For arbitrary  $ t \geq 0 $, $ W $ has the following expansion:
%  \begin{equation}\label{supeq:Q_Wiener}
%     W(t) = \sum_{j=1}^{\infty} \sqrt{\lambda_{j}} \beta_{j}(t) e_{j},
%  \end{equation}
% \noindent where  $ \beta_{j}(t)  $  are real valued brownian motions that are mutually independent on $ (\Omega, \F, \mathbb{P}). $% defined as:
%  %\begin{equation}\label{eq:Brownian_Motion}
%      %\beta_{j}(t) = \frac{1}{\sqrt{\lambda_{j}}} \langle W(t), e_{j} \rangle.
% % \end{equation}  
%  \end{enumerate}
% \end{prop}

\subsection{Relative Entropy and Free Energy Dualities in Hilbert Spaces}
\label{supsec:Free_Energy_Relative_Entropy}

In this section we provide the relation between free energy and relative entropy. This connection is valid for general probability measures, including  measures defined on path spaces induced by infinite-dimensional stochastic systems.  In what follows, $L^{p}$ ($1\leq p <\infty$) denotes the standard $L^p$ space of measurable functions and $\Pc$ denotes the set of probability measures.

\begin{definition}{\label{def:Free_Energy} \textit{(Free Energy)} Let  $   \calL \in \Pc $ a probability measure on a sample space $\Omega$, and consider a measurable function   $J: L^{p} \to  \Rb_{+}$. Then the following term:
  \begin{equation}
V  :=\frac{1}{\rho} \log_e   \int_{\Omega} \exp(\rho J )   \rd \calL (\omega),
  \end{equation}
 \noindent     is called the {\it free energy}\footnote{The function  $ \log_{e}  $ denotes the natural  logarithm.} of  $ J   $    with respect to $   \calL $ and  $ \rho \in \Rb$.}  
 \end{definition}  

  \begin{definition}{\label{def:Entropy} \textit{(Generalized Entropy)} Let  $   \calL ,\tilde{\calL} \in \Pc$, then the relative entropy of $  \tilde{\calL} $  with respect to $   \calL $ is defined as: 
\[    S\left( \tilde{\calL }|| \calL  \ \right) := \left\{
\begin{array}{l l}
  -\int_{\Omega}   \frac{    \rd \tilde{\calL}(\omega) }{\rd \calL(\omega)}  \log_e  \frac{ \rd \tilde{\calL}(\omega)}{\rd \calL (\omega)}         \rd \calL(\omega),
  \mbox{if $\tilde{\calL}<<\calL $},  \\
  +\infty,  \quad \mbox{otherwise},\\ \end{array} \right. \]
where ``$<<$'' denotes absolute continuity of $\tilde{\calL}$ with respect to $\calL$. We say that $ \tilde{\calL} $ is \textit{absolutely continuous} with respect to $ \calL $ and we write  $\tilde{\calL}<<\calL $ if  $ \calL(B) = 0 \Rightarrow \tilde{\calL}(B) = 0, ~ \forall B \in \F$. }
 \end{definition}  
 The  free energy and relative entropy relationship is expressed by the following theorem:
\begin{theorem} \textit{  Let  $( \Omega, {\F}) $  be a measurable space. Consider  $ \calL , \tilde{\calL}  \in  \Pc $ and     definitions \ref{def:Free_Energy}, \ref{def:Entropy}.  Under the assumption that $\tilde{\calL}<<\calL $,  the following inequality holds:}
    \begin{align} \label{supeq:Legendre}
    & - \frac{1}{\rho}   \log_e \Eb_{\calL} \bigg[ \exp( -\rho {J} )  \bigg]  \leq \bigg[    \Eb_{\tilde{\calL}}\left({J} \right)  -\frac{1}{\rho} S \left( \tilde{\calL}\hspace{0.05cm} \big|\big|  \calL \right)  \bigg] ,
    \end{align}   
 \noindent  \textit{where $ \mathbb{E}_{\calL}, \mathbb{E}_{\tilde{\calL}}  $  denote expectations under probability measures  $ \calL$, $\tilde{\calL} $  respectively. Moreover, $ \rho \in \Rb_{+}$ and $ J :  L^{p} \to  \Rb_{+} $.  The inequality in \eqref{supeq:Legendre} is the so called   Legendre Transform.   }  
\end{theorem}
 By defining the free energy  as temperature $  T = \frac{1}{\rho} $,  the Legendre transformation has the form: 
\begin{equation}\label{eq:Legendre_StatMech}
V  \leq E  - T S,
\end{equation}
and the equilibrium probability measure has the classical form: 
\begin{equation}\label{supeq:Gibbs}
\rd \calL^{*}(\omega) = \frac{\exp( - \rho J) \rd \calL(\omega)}{\int_{\Omega} \exp( - \rho J)  \rd \calL(\omega)},
\end{equation}
To verify the optimality of $\calL^*$, it suffices to substitute \eqref{supeq:Gibbs} in \eqref{supeq:Legendre} and show that the inequality collapses to an equality \cite{entropy_2015}.  The statistical physics  interpretation of inequality \eqref{eq:Legendre_StatMech} is that, maximization of entropy results in reduction of the available energy. At the thermodynamic equilibrium the entropy reaches its maximum and $V  = E  - T S$.

%%%%%%%%%%%%%%%%%%%%%%%%%%%%%%%%%%%%%%%%%%%%%%%%%%%%%%%%%%%%%%%%%%%%%%%%%%%%%%%%%%%%%%%%%%%%%%%%%%NEW_SECTION%%%%%%%%%%%%%%%%%%%%%%%%%%%%%%%%%%%%%%%%%%%%%%%%%%%%%%%%%%%%%%%%%%%%%%%%%%%%%%%%%%%%%%%%%%%%%%%%%%%%%%%%%%%

\subsection{A Girsanov Theorem for SPDEs}\label{supsec:girsanov}

\begin{theorem}[Girsanov] \label{girs} Let $\Omega$ be a sample space with a $\sigma$-algebra $\mathcal{F}$. Consider the following $H$-valued stochastic processes:
\begin{align}
\rd X&=\big(\A X+F(t, X)\big) \rd t+G(t, X)\rd W(t), \label{X}\\
\rd\tilde{X}&=\big(\A \tilde{X}+F(t, \tilde{X})\big)\rd t+\tilde{B}(t, \tilde{X})\rd t+G(t, \tilde{X})\rd W(t), \label{X_tilde}
\end{align}
where $X(0)=\tilde{X}(0)=x$ and $W\in U$ is a cylindrical Wiener  process with respect to measure $\mathbb{P}$. Moreover, for each $\Gamma\in C([0,T]; H)$, let the {\it law} of $X$ be defined as $\mathcal{L}(\Gamma):=\mathbb{P}(\omega\in\Omega|X(\cdot,\omega)\in\Gamma)$. Similarly, the law of $\tilde{X}$ is defined as $\tilde{\calL}(\Gamma):=\mathbb{P}(\omega\in\Omega|\tilde{X}(\cdot,\omega)\in\Gamma)$. Assume
\begin{equation}
    \label{psi_assum1}
    \mathbb{E}_{\mathbb{P}}\big[e^{\frac{1}{2}\int_{0}^{T}||\psi(t)||^2\mathrm dt}\big]<+\infty,
\end{equation}
where
\begin{equation}
    \psi(t):=G^{-1}\big(t, X(t)\big)\tilde{B}\big(t, X(t)\big)\in U_{0}.
\end{equation}
Then
\begin{equation} \label{eq:Lg}
\begin{split}
\tilde{\calL}(\Gamma)=  \textstyle\mathbb{E}_{\mathbb{P}}\big[\exp\big(\int_{0}^{T}\langle\psi(s), dW(s)\rangle_{U}-\frac{1}{2}\int_{0}^{T}||\psi(s)||_{U}^{2}ds\big)|X(\cdot)\in\Gamma\big].
\end{split}
\end{equation}

% Here, we write for brevity $\tilde{\calL}(\omega)\equiv\tilde{\calL}(\tilde{X}(\cdot,\omega)\in\Gamma)$.
\end{theorem}
\begin{proof}
Define the process:
\begin{equation}
\label{w_hat}
\hat{W}(t):=W(t)-\int_{0}^{t}\psi(s)\rd s.
\end{equation}
Under the assumption in \eqref{psi_assum1}, $\hat{W}$ is a cylindrical Wiener process with respect to a measure $\Qb$ determined by:
\begin{equation}
\label{girsanov_measure}
\begin{split}
\rd \Qb (\omega)&=\exp\big(\int_{0}^{T}\langle\psi(s),\rd W(s)\rangle_{U}-\frac{1}{2}\int_{0}^{T}||\psi(s)||_{U}^{2}\rd s\big)\rd\mathbb{P} \\ &=\exp\big(\int_{0}^{T}\langle\psi(s),\rd \hat{W}(s)\rangle_{U}+\frac{1}{2}\int_{0}^{T}||\psi(s)||_{U}^{2}\rd s\big)\rd\mathbb{P}.
\end{split}
\end{equation}
The proof for this result can be found in \cite[Theorem 10.14]{da1992stochastic}. Now, using \eqref{w_hat}, \eqref{X} will be rewritten as:
\begin{align}
\rd X&=\big(\A  X+F(t, X)\big)\rd t+G(t, X)\rd W(t) \label{X_new0}  \\
     &=\big(\A  X+F(t, X)\big)\rd t+B(t, X)\rd t+G(t, X)\rd\hat{W}(t) \label{X_new1}
\end{align}
Notice that  the  SPDE in \eqref{X_new1} has the same form as \eqref{X_tilde}. Therefore, under the introduced measure $\Qb$ and noise profile $\hat{W}$, $X(\cdot, \omega)$ becomes equivalent to $\tilde{X}(\cdot, \omega)$ from \eqref{X_tilde}. Conversely, under measure $\mathbb{P}$, \eqref{X_new0} (or \eqref{X_new1}) behaves as the original system in \eqref{X}. In other words, \eqref{X} and \eqref{X_new1} describe the same system on $(\Omega, \mathcal{F}, \mathbb{P})$. From the uniqueness of solutions and the aforementioned reasoning, one has:
\[\Pb\big(\{\tilde{X}\in\Gamma\}\big) = \Qb\big(\{X\in\Gamma\}\big).\]
The result follows from \eqref{girsanov_measure}.
\end{proof}

%%%%%%%%%%%%%%%%%%%%%%%%%%%%%%%%%%%%%%%%%%%%%%%%%%%%%%%%%%%%%%%%%%%%%%%%%%%%%%%%%%%%%%%%%%%%%%%%%%NEW_SECTION%%%%%%%%%%%%%%%%%%%%%%%%%%%%%%%%%%%%%%%%%%%%%%%%%%%%%%%%%%%%%%%%%%%%%%%%%%%%%%%%%%%%%%%%%%%%%%%%%%%%%%%%%%%

\subsection{Derivation of Variational Optimal Control of Spatio-Temporal Processes} \label{supsec:derivation}
Under the open loop parameterization $\calU(\vx,t) =  \vm(\vx)^{\rT} \vu(t)  $, the problem takes the form:
     \begin{align*}
    \vu^{*}  &=   \argmin   \bigg[   \int_{C}   \ \log_e  \frac{ \rd \calL^{*}(\mathsf{x})}{\rd \tilde{\calL}(\mathsf{x})}  \rd \calL^{*}(\mathsf{x}) \bigg] =    \argmin   \bigg[   \int_{C}   \ \log_e  \frac{ \rd \calL^{*}(\mathsf{x})}{\rd \calL(\mathsf{x})} \frac{ \rd \calL(\mathsf{x})}{\rd \tilde{\calL}(\mathsf{x})}  \rd \calL^{*}(\mathsf{x}) \bigg]. 
%                  & =    \argmin   \bigg[   \int_{\Omega}    \log_e  \frac{ \rd \rQ_{h}^{*}(\omega)}{\rd \rP_{h}(\omega)}  + \log_{e} \frac{ \rd \rP_{h}(\omega)}{\rd \rQ_{h}(\omega)}  \rd \rQ_{h}^{*}(\omega) \bigg]  \\
 %             & =   \argmin   \bigg[   \int_{\Omega}     \log_{e} \frac{ \rd \rP_{h}(\omega)}{\rd \rQ_{h}(\omega)}  \rd \rQ_{h}^{*}(\omega) \bigg] \\
      %        & =  \argmin  \Eb_{\calL^{*}}  \bigg[     \log_{e} \frac{ \rd \calL(\omega)}{\rd \tilde{\calL}(\omega)}  \bigg] 
  %            & =  \argmin \bigg[    \int_{\Omega} \int_{0}^{T} \int_{O}    \vm(\vx)^{\rT} \vu(t) \rd \hat{\omega}(\vx,t)
  \end{align*} 
By using the change of measures in \cref{eq:radon_param} of the main text, minimization of the last expression is equivalent to the minimization of the  expression: 
     \begin{align*}
 \Eb_{\calL^{*}}  \bigg[     \log_{e} \frac{ \rd \calL(\mathsf{x})}{\rd \tilde{\calL}(\mathsf{x}
 )}  \bigg]   &=   -\sqrt{\rho} \Eb_{\calL^{*}}  \bigg[  \int_{0}^{T}\vu(t)^{\top}\bar{ \vm}(t)\bigg] +\frac{1}{2} \rho\Eb_{\calL^{*}}  \bigg[  \int_{0}^{T}\vu(t)^{\top}\vM\vu(t)\rd t\bigg].
  \end{align*} 
 Since we apply the control in discrete time instances, it suffices to consider the class of step functions $\vu_i$, $i=0, \dots, L-1$ that are constant over fixed-size intervals $[t_i, t_{i+1}]$ of length $\Delta t$:
    \begin{align*}
  \Eb_{\calL^{*}}  \bigg[ \log_{e} \frac{ \rd \calL(\mathsf{x})}{\rd \tilde{\calL}(\mathsf{x})}  \bigg]  &=  -\sqrt{\rho} \sum_{i=0}^{L-1}\vu_{i}^{\top}\Eb_{\calL^{*}}  \bigg[  \int_{t_{i}}^{t_{i+1}}\bar{ \vm}(t)\bigg]  +\frac{1}{2}\rho\sum_{i=0}^{L-1}\vu_{i}^{\top}\vM\vu_{i}\Delta t,
  \end{align*} 
 where we have used the fact that $\vM$ is constant with respect to time. Due to the symmetry of $\vM$, minimization of the expression above with respect to  $ \vu_{i}$ results in: 
  \begin{equation}\label{eq:optimalcontrol}
      \vu_{i}^{*} =  \frac{1}{\sqrt{\rho} \Delta t  }\vM 
      ^{-1}  \Eb_{\calL^{*}}  \bigg[\int_{t_{i}}^{t_{i+1}}\bar{ \vm}(t)\bigg].
  \end{equation}
 Since we cannot sample directly from the optimal measure $\calL^{*}$, we need to
express the above expectation with respect to
the measure induced by controlled dynamics,  $\calL^{(i)}$. We can then directly sample
controlled trajectories based on $\calL^{(i)}$ and approximate the optimal control trajectory. The change
in expectation is achieved by applying the Radon-Nikodym derivative. These so called importance sampling steps are as follows. First define $W^{(i)}$ in a similar fashion to \cref{w_hat}, as:
\begin{equation} \label{w_hat_i}
    W^{(i)}(t) := W(t) - \int_0^t \sqrt{\rho} \calU^{(i)}(s) \rd s.
\end{equation}
Similar to \cref{X_new1}, one can rewrite the uncontrolled dynamics as
\begin{align}\label{eq:SPDEs_NoControl_Iter_i}
\begin{split}
\rd X &= \big(\A X   + F(t, X) \big)  \rd t +   \frac{1}{\sqrt{\rho}} G(t, X) \rd W(t)  \\ 
&=\big( \A X   + F(t, X) \big) \rd t + G(t, X) \big( \calU^{(i)}  \rd t  + \frac{1}{\sqrt{\rho}} \rd W^{(i)}(t) \big).
\end{split}
\end{align}

Under the open loop parameterization $\calU(t)(\vx) =\vm(\vx)^{\top} \vu_j$, where  $\vu_j$ are step functions on each interval $[t_j, t_{j+1}]$, the change of measures becomes
\begin{align}
\label{radon_i}
\begin{split}
\frac{ \rd \calL}{ \rd \calL^{(i)} } = \exp\bigg(-\sqrt{\rho}\sum_{k=0}^{L-1}\vu_k^{(i)\top}\int_{t_k}^{t_{k+1}}\bar{ \vm}^{(i)}(t)-\rho\frac{1}{2}\sum_{k=0}^{L-1}\vu_k^{(i)\top}\vM\vu_k^{(i)}\Delta t\bigg),
\end{split}
\end{align}
where
\begin{align}
    \bar{ \vm}^{(i)}(t) :=& \bigg[\langle m_{1},\rd W^{(i)}(t)\rangle_{U},...,\langle m_{N},\rd W^{(i)}(t)\rangle_{U}\bigg]^{\top}\in\mathbb{R}^{N}.\label{eq:small_m_i}
\end{align}
One can alternatively write this as
\begin{equation*}
\begin{split}
\bigg(\int_{t_j}^{t_{j+1}}\bar{ \vm}^{(i)}(t)\bigg)_{l}&=\int_{t_j}^{t_{j+1}} \big\langle m_l,\rd W^{(i)}(t)\big\rangle_{U}=\int_{t_j}^{t_{j+1}}\big\langle m_l,\rd W(t)-\sqrt{\rho}\calU^{(i)}(t)\rd t\big\rangle_{U}\\
&=\int_{t_j}^{t_{j+1}}\big\langle m_l,\rd W(t)\big\rangle_{U}-\sqrt{\rho}\bigg[\big\langle m_{l},m_{1}\big\rangle_{U},...,\big\langle m_{l},m_{N}\big\rangle_{U}\bigg]\vu_{j}^{(i)}\Delta t.
\end{split}
\end{equation*}
It follows that:
\begin{equation} \label{eq:small_m_i_redef}
    \int_{t_j}^{t_{j+1}}\bar{ \vm}^{(i)}(t)= \int_{t_j}^{t_{j+1}}\bar{ \vm}(t)-\sqrt{\rho}\Delta t\vM \vu_{j}^{(i)}.
\end{equation}

In order to derive the iterative scheme, we perform one step of importance sampling and express the associated expectations with respect the measure induced by the  controlled SPDE in \cref{eq:SPDEs_Control} of the main text. Let us begin by modifying \cref{eq:optimalcontrol} via  the appropriate change of measures from \eqref{radon_i}, as well as \eqref{w_hat_i}:
  \begin{align}
      \vu_{j}^{i+1} =  \frac{1}{\sqrt{\rho} \Delta t  }\vM 
      ^{-1}  \int_\Omega \bigg[ \frac{\rd \calL^{*}}{\rd \calL}\frac{\rd \calL}{\rd \calL^{(i)}} \int_{t_{i}}^{t_{i+1}}\bar{ \vm}(t)\bigg]  \rd \calL^{(i)}  &=  \frac{1}{\sqrt{\rho} \Delta t  }\vM 
      ^{-1}  \int \bigg[\frac{\exp(-\rho J)}{\Eb_{\calL} \big[\exp(-\rho J) \big]} \frac{\rd \calL}{\rd \calL^{(i)}}  \int_{t_{i}}^{t_{i+1}}\bar{ \vm}(t)\bigg] \rd \calL^{(i)} \nonumber \\
      &=  \frac{1}{\sqrt{\rho} \Delta t  }\vM 
      ^{-1}  \int \bigg[\frac{\exp(-\rho J)}{\Eb_{\calL^{(i)}} \big[\frac{\rd \calL}{\rd \calL^{(i)}}\exp(-\rho J) \big]} \frac{\rd \calL}{\rd \calL^{(i)}} \int_{t_{i}}^{t_{i+1}}\bar{ \vm}(t)\bigg]\rd \calL^{(i)} \nonumber \\
      &=  \frac{1}{\sqrt{\rho} \Delta t  }\vM 
      ^{-1} \Eb_{\calL^{(i)}} \bigg[\frac{\exp(-\rho J^{(i)})}{\Eb_{\calL^{(i)}} \big[\exp(-\rho J^{(i)}) \big]} \int_{t_{i}}^{t_{i+1}}\bar{ \vm}(t)\bigg], \label{eq:u_j_iplus1_is}
  \end{align}
One can reorder \cref{eq:small_m_i_redef} as
\begin{equation} \label{eq:small_m_i_redef1}
    \int_{t_j}^{t_{j+1}}\bar{ \vm}(t)= \int_{t_j}^{t_{j+1}}\bar{ \vm}^{(i)}(t)+\sqrt{\rho}\Delta t\vM \vu_{j}^{(i)}.
\end{equation}
and plug it into \cref{eq:u_j_iplus1_is} to yield:
\begin{align}
    \vu_{j}^{i+1} &= \frac{1}{\sqrt{\rho} \Delta t  }\vM^{-1} \Eb_{\calL^{(i)}} \bigg[\frac{\exp(-\rho J^{(i)})}{\Eb_{\calL^{(i)}} \big[\exp(-\rho J^{(i)}) \big]} \int_{t_j}^{t_{j+1}}\bar{ \vm}^{(i)}(t)+\sqrt{\rho}\Delta t\vM \vu_{j}^{(i)}\bigg] \nonumber \\
    &= \vu_j^{(i)} + \frac{1}{\sqrt{\rho} \Delta t  }\vM^{-1} \Eb_{\calL^{(i)}} \bigg[\frac{\exp(-\rho J^{(i)})}{\Eb_{\calL^{(i)}} \big[\exp(-\rho J^{(i)}) \big]} \int_{t_j}^{t_{j+1}}\bar{ \vm}^{(i)}(t)\bigg],
\end{align}
which is equivalent to \cref{eq:Iterative_optimalvariation1} in the main text with $J^{(i)}$ defined by \cref{eq:J_i} in the main text.
% The result is equation \eqref{eq:optimalcontrol}.

%  \begin{align}\begin{split} \label{eq:Inf_Kolmogorov1}
%     - \partial_{t} \psi\big(t,X(t)\big) &= -\rho \ell\big(t, X(t)\big)   \psi\big(t,X(t)\big)  + \big\langle \psi_{X}, \A X(t) + F\big(X(t)\big) \big\rangle +  \frac{1}{2} \tr\Big[ \psi_{XX} (B Q^{\frac{1}{2}})  ( B Q{}^{\frac{1}{2}})^{*}  \Big]  \\
%     & \quad +  g\big(t,X(t)\big).
%   \end{split}
% \end{align} 

%%%%%%%%%%%%%%%%%%%%%%%%%%%%%%%%%%%%%%%%%%%%%%%%%%%%%%%%%%%%%%%%%%%%%%%%%%%%%%%%%%%%%%%%%%%%%%%%%%NEW_SECTION%%%%%%%%%%%%%%%%%%%%%%%%%%%%%%%%%%%%%%%%%%%%%%%%%%%%%%%%%%%%%%%%%%%%%%%%%%%%%%%%%%%%%%%%%%%%%%%%%%%%%%%%%%%

\subsection{Feynman-Kac for Spatio-Temporal Diffusions: From Expectations to Hilbert Space PDEs}\label{supbsec:Feynman-Kac}
\begin{lemma}{(Infinite Dimensional Feynman-Kac):} Define $\psi:[t_0,T]\times H \rightarrow \mathbb{R}$ as the conditional  expectation: 
\begin{align}\label{eq:Feynman_Kac}
\psi(t, X) :=\Eb_{\calL} \bigg[ \exp\Big(  -\rho {J \big( \Xb^{T}_{t,X} \big)  } \Big) \bigg| \F_{t} \bigg] +\Eb_{\calL} \bigg[  \int_{t}^{T} g(X,t) \exp\Big(  -\rho {\Phi \big( \Xb^{s}_{t,X} \big) } \Big) \rd s \bigg|  \F_{t} \bigg],
\end{align}
  
\noindent evaluated on stochastic trajectories $ \Xb^{T}_{t,X}  $  generated by the infinite dimensional stochastic systems in eqs. (2) and (3) of the main text and  $\rho \in \Rb_{+}$.  The trajectory dependent  terms  $\Phi \big( \Xb^{T}_{t,X}  \big): L^{p} \to \Rb_{+}$ and $J  \big( \Xb^{T}_{t,X} \big) : L^{p} \to \Rb_{+}$ are defined as follows:  
  \begin{align}\begin{split}\label{eq:StateCost}
     \Phi \big( \Xb^{s}_{t,X}  \big) &=   \int_{t}^{s} \ell\big(\tau,X(\tau)\big) \rd \tau, \\
      J  \big( \Xb^{T}_{t,X}  \big)  &=  \phi(T,X) + \Phi \big( \Xb^{T}_{t,X}  \big).\end{split}
  \end{align}  
Also, let $\psi(t,X) \in C_{b}^{1,2}([0,T] \times H)$. Then the function $\psi(t,X)$ satisfies the following equation:
 \begin{align}\begin{split} \label{eq:Inf_Kolmogorov1}
    - \partial_{t} \psi\big(t,X(t)\big) &= -\rho \ell\big(t, X(t)\big)   \psi\big(t,X(t)\big)  + \big\langle \psi_{X}, \A X(t) + F\big(X(t)\big) \big\rangle +  \frac{1}{2} \tr\Big[ \psi_{XX} (B Q^{\frac{1}{2}})  ( B Q{}^{\frac{1}{2}})^{*}  \Big]  \\
    & \quad +  g\big(t,X(t)\big).
  \end{split}
\end{align} 
\end{lemma}

{\begin{proof}  The proof starts with the expectation in \eqref{eq:Feynman_Kac} which is an expectation conditioned on the filtration  $\F_{t}$. To keep the notation short we will drop the dependencies on $t$ and $X(t)$, and will write $\phi_{T} = \phi\big(T, X(T)\big)$,  $\ell_{t} = \ell\big(t, X(t)\big)$, and $g_{t} = g\big(t,X(t)\big)$. We split the integrals inside the expectations to write:
  \begin{align*}
\psi(t,X) &=\Eb_{\calL} \bigg[ \exp\bigg(  -\rho \phi_{T}  -  \rho  \int_{t}^{T} \ell_{\tau} \rd \tau    \bigg) \bigg|  \F_{t} \bigg] +\Eb_{\calL}  \bigg[  \int_{t}^{T} g_{s} \exp{\bigg(  -\rho \int_{t}^{s} \ell_{\tau} \rd \tau    \bigg)   }\rd s    \bigg|  \F_{t}  \bigg] \\
     &=  \Eb_{\calL}  \bigg[ \exp\bigg(  -\rho \phi_{T} -  \rho  \int_{t+\delta t}^{T} \ell_{\tau} \rd \tau    \bigg)\exp (- \int_{t}^{t+\delta t} \ell_{\tau} \rd \tau)  \bigg|  \F_{t} \bigg]  \nonumber\\
     &\quad + \Eb_{\calL} \bigg[  \int_{t}^{t+\delta t} g_{s} \exp{\bigg(  -\rho \int_{t}^{s} \ell_{\tau} \rd \tau    \bigg)   }\rd s   \bigg| \F_{t} \bigg] + \Eb_{\calL}  \bigg[  \int_{t+\delta t}^{T} g_{s} \exp{\bigg(  -\rho \int_{t}^{s} \ell_{\tau} \rd \tau    \bigg)   }\rd s   \bigg|  \F_{t} \bigg] \
     \end{align*}
    
%  We  choose $  \delta t \to 0 $  and therefore substitute it with $ \rd t $.    
%%     \scriptsize
%      \begin{align*}
%&\psi(X,t) =\\
%     &=  \Eb \bigg[ \exp\big(  -\rho \phi_{T} -  \rho  \int_{t+\rd t}^{t_f} \ell_{\tau} \rd \tau    \big) \exp (- \int_{t}^{t+\rd t} \ell_{\tau} \rd \tau) \bigg|  \F_{t} \bigg] \nonumber\\
%     &+\Eb \bigg[  \int_{t}^{t+\rd t} g_{s} \exp{\big(  -\rho \int_{t}^{s} \ell_{\tau} \rd \tau    \big)   }\rd s  \bigg|  \F_{t} \bigg] \nonumber\\
%     &+\Eb \bigg[  \int_{t+\rd t}^{t_f} g_{s} \exp\big(  -\rho \int_{t}^{t+ \rd t} \ell_{\tau} \rd \tau    \big) \\
%     & \times  \exp{\big(  -\rho \int_{t+\rd t}^{s} \ell_{\tau} \rd \tau}    \big)  \rd s   \bigg|  \F_{t} \bigg]. \\
%       \end{align*} 
  
  By using the law of iterated expectations between the two sub-sigma algebras $ \F_{t} \subseteq \F_{t+ \delta t} $  we have that:
  \small
  \begin{align*}
  \psi(t, X) &= \Eb_{\calL} \Bigg[  \Eb_{\calL}  \bigg[ \exp\bigg(  -\rho \phi_{T} -  \rho  \int_{t+\delta t}^{T} \ell_{\tau} \rd \tau    \bigg) \exp \bigg(- \int_{t}^{t+\delta t} \ell_{\tau} \rd \tau \bigg) \bigg|  \F_{t+\delta t} \bigg] \bigg|  \F_{t}  \bigg]  \nonumber\\
  & \quad + \Eb_{\calL}  \bigg[   \int_{t}^{t+\rd t} g_{s} \exp{\bigg(  -\rho \int_{t}^{s} \ell_{\tau} \rd \tau    \bigg) }\rd s  \bigg|  \F_{t} \Bigg] \\
  &\quad + \Eb_{\calL}  \bigg[   \Eb_{\calL}   \bigg[ \int_{t+\rd t}^{T} g_{s} \exp\bigg(  -\rho \int_{t}^{t+ \rd t} \ell_{\tau} \rd \tau    \bigg)  \exp{\bigg(  -\rho \int_{t+\rd t}^{s} \ell_{\tau} \rd \tau}    \bigg)  \rd s   \bigg|  \F_{t+\delta t} \bigg]  \bigg|  \F_{t}  \bigg].
  \end{align*} \normalsize   
      Next we use the fact that  the conditioning on the filtration $ \F_{t+\delta t} $ results in the following equality:     
    \begin{align*}
        &\Eb_{\calL} \bigg[   \Eb_{\calL}   \bigg[ \exp\bigg(  -\rho \phi_{T} -  \rho  \int_{t+\delta  t }^{T} \ell_{\tau} \rd \tau    \bigg) \exp \bigg(- \int_{t}^{t+\delta  t } \ell_{\tau} \rd \tau\bigg) \bigg|  \F_{t+\delta  t } \bigg] \bigg|  \F_{t}  \bigg]  \\
        & =  \Eb_{\calL} \bigg[  \exp \bigg(- \int_{t}^{t+\delta  t } \ell_{\tau} \rd \tau\bigg)    \Eb_{\calL}   \bigg[ \exp\bigg(  -\rho \phi_{T} -  \rho  \int_{t+\delta  t}^{T} \ell_{\tau} \rd \tau    \bigg) \bigg|  \F_{t+\delta t} \bigg] \bigg|  \F_{t}  \bigg]
     \end{align*}  
         \normalsize
      By further using this property of independence we have:
          \small
     \begin{align*}
      \psi(t,X) &=  \Eb_{\calL} \bigg[  \exp \bigg(- \rho \int_{t}^{t+\delta  t } \ell_{\tau} \rd \tau \bigg)  \Eb_{\calL}  \bigg[ \exp\bigg(  -\rho \phi_{T} -  \rho  \int_{t+\delta  t}^{T} \ell_{\tau} \rd \tau    \bigg)  \bigg|  \F_{t+\delta t} \bigg] \bigg|  \F_{t}  \bigg]  \nonumber\\
      &\quad + \Eb_{\calL}  \bigg[  \int_{t}^{t+\rd t} g_{s} \exp{\bigg(  -\rho \int_{t}^{s} \ell_{\tau} \rd \tau    \bigg)   }\rd s  \bigg|  \F_{t} \bigg] \\
      &\quad + \Eb_{\calL}  \bigg[ \exp\bigg(  -\rho \int_{t}^{t+\delta  t } \ell_{\tau} \rd \tau    \bigg) \bigg] \Eb_{\calL}  \bigg[ \int_{t+\delta  t }^{T} g_{s}    \exp{\bigg(  -\rho \int_{t+\delta  t  }^{s} \ell_{\tau} \rd \tau}    \bigg)  \rd s   \bigg|  \F_{t+\delta t} \bigg]  \bigg| \F_{t}  \bigg] \\
      &=  \Eb_{\calL} \bigg[  \exp \bigg( - \rho \int_{t}^{t+\delta  t } \ell_{\tau} \rd \tau \bigg)  \psi(t+ \delta  t , X(t+\delta  t ) ) \bigg| \F_{t}  \bigg] +\Eb_{\calL}  \bigg[  \int_{t}^{t+\delta  t } g_{s} \exp{\bigg(  -\rho \int_{t}^{s} \ell_{\tau} \rd \tau    \bigg)   }\rd s  \bigg|  \F_{t} \bigg] \nonumber
     \end{align*}               \normalsize     
  The last expression provides the backward propagation of the  $ \psi\big(t,X(t)\big) $ by employing a expectation over $ \psi\big(t+\delta t,X(t + \delta t)\big) $.  To get the backward  deterministic Kolmogorov equations for the infinite dimensional case we  subtract the term $  \small{\Eb \bigg[ \psi\big(t +\delta  t ,X(t+\delta  t )\big) \bigg|  \F_{t} \bigg] }$  from both sides:   
\small
 \begin{align*}
  -  \Eb_{\calL}  \bigg[  \psi\big(t+ \delta t, X(t+\delta t) \big)  -    \psi\big(t, X(t)\big) \bigg|  \F_{t} \bigg]    &= \Eb_{\calL}  \bigg[   \bigg\lbrace\exp\bigg( -\rho   \int_{t}^{t+\delta  t } \ell_{\tau}  \rd \tau\bigg) -1 \bigg\rbrace \psi\big(t + \delta  t ,X(t+\delta  t )\big)  \bigg|  \F_{t} \bigg] \\
  & \quad +  \Eb_{\calL}  \bigg[  \int_{t}^{t+\delta  t } g_{s} \exp{\bigg(  -\rho \int_{t}^{s} \ell_{\tau} \rd \tau    \bigg)   }\rd s  \bigg|  \F_{t} \bigg].
\end{align*}\normalsize
Next we take the limit as $ \delta t \to 0 $ we have:
\small
\begin{align*}
    &- \lim_{\delta  t  \to 0}  \Eb_{\calL}  \bigg[  \psi\big(t+ \delta t, X(t+\delta t)\big)  -    \psi\big(t, X(t)\big)  \bigg|  \F_{t} \bigg]  \\
    & = \lim_{\delta  t \to 0}  \Eb_{\calL}  \bigg[    \bigg(\exp( -\rho   \int_{t}^{t+\delta  t } \ell_{\tau}  \rd \tau) -1 \bigg)\psi\big(t +\delta  t ,X(t+\delta  t )\big)  \bigg|  \F_{t} \bigg]  +  \lim_{\delta  t  \to 0}  \Eb_{\calL}  \bigg[ \int_{t}^{t+\delta  t } g_{s} \exp{\big(  -\rho \int_{t}^{s} \ell_{\tau} \rd \tau    \big)   }\rd s  \bigg|  \F_{t} \bigg].
\end{align*}\normalsize

Thus we have to compute three terms. We employ the Lebegue dominated convergence theorem to pass the limit inside the expectations:
\begin{equation}
 -  \lim_{\delta  t  \to 0}   \Eb_{\calL}  \bigg[  \psi\big(t+ \delta t, X( t+\delta t)\big)  -    \psi\big(t,X(t)\big)  \bigg|  \F_{t} \bigg] =  \Eb_{\calL}  \bigg[ \rd  \psi \bigg|  \F_{t} \bigg]
\end{equation}
By using the It\^{o} differentiation rule \cite[Theorem 4.32]{da1992stochastic} for the case of infinite dimensional stochastic systems we have that:
 \begin{align*}\begin{split}
  \Eb_{\calL}  \bigg[ \rd \psi\big(t, X(t)\big) \bigg|  \F_{t} \bigg] =    \partial_{t} \psi \big(t,X(t)\big)  \rd  t   +    \big\langle \psi_{X}, \A X(t) + F\big(X(t)\big) \big\rangle \rd  t  + \frac{1}{2} \tr\Big[\psi_{XX} (B Q^{\frac{1}{2}})  ( B Q{}^{\frac{1}{2}})^{*}  \Big] \rd t  \end{split}
\end{align*}
The next term is

\begin{align*}
\begin{split}
  \lim_{\delta  t   \to 0} \Eb_{\calL}  \bigg[   \bigg(\exp( -\rho   \int_{t}^{t+\delta  t } \ell_{\tau}  \rd \tau) -1 \bigg)\psi\big(t +\delta  t ,X(t+\delta  t )\big)  \bigg|  \F_{t} \bigg]  &=  - \Eb_{\calL}  \bigg[  \ell_{t} \psi\big(t ,X(t)\big)  \bigg|  \F_{t} \bigg] \\
  &= - \rho \ell\big(t,X(t)\big) \psi\big(t ,X(t)\big)  \rd t
\end{split}
\end{align*}
\normalsize
The third term is 
 \begin{align*}\begin{split}
\lim_{\delta  t \to 0}  \Eb_{\calL}  \bigg[ \int_{t}^{t+\delta  t } g_{s} \exp{\bigg(  -\rho \int_{t}^{s} \ell_{\tau} \rd \tau    \bigg)   }\rd s  \bigg|  \F_{t} \bigg] =  \Eb_{\calL}  \bigg[ g\big(t,X(t)\big) \delta  t  \bigg|  \F_{t}  \bigg] =  g\big(t,X(t)\big) \rd t
\end{split}
\end{align*}

% \begin{align*}
%  &  -    \psi_{t}(t,X(t)) \rd t    -    \langle \psi_{X}, \calA X(t) + \calF(t,X(t)) \rangle \rd t \\
%  & - \frac{1}{2} \tr(\psi_{XX} (B Q^{\frac{1}{2}})  ( B Q{}^{\frac{1}{2}})^{*}  ) \rd t    \\
%  &  =  (\exp( -\rho \ell  \rd t) -1 ) \Eb_{\rP(X^{'}|X)} \bigg[ \psi(t + \rd t,X(t+\rd t))  \bigg|  \F_{t} \bigg]  \\
%  &+  g(t,X(t)) \rd t.
%\end{align*}

Combining the three terms above, we have shown that $\psi\big(t,X(t)\big)$ satisfies the backward Kolmogorov equation for the case of the infinite dimensional stochastic system in \cref{eq:SPDEs_Control} of the main text. 
\end{proof}}

%%%%%%%%%%%%%%%%%%%%%%%%%%%%%%%%%%%%%%%%%%%%%%%%%%%%%%%%%%%%%%%%%%%%%%%%%%%%%%%%%%%%%%%%%%%%%%%%%%NEW_SECTION%%%%%%%%%%%%%%%%%%%%%%%%%%%%%%%%%%%%%%%%%%%%%%%%%%%%%%%%%%%%%%%%%%%%%%%%%%%%%%%%%%%%%%%%%%%%%%%%%%%%%%%%%%%

\subsection{Connections to Stochastic Dynamic Programming}\label{supsec:Connections_DP}
%\textbf{You have some symbols that you haven't defined. Like $\mathbb{X}, \otimes$. }

In this section we show the connections between stochastic dynamic programming and the free energy.  Before proceeding, let $C_{b}^{k,n}([0,T] \times H)$  denote the space of all functions $\xi:[0, T] \times H \to \Rb^{1}$ that are $k$ times continuously \textit{Fr\'{e}chet}  differentiable with respect to time $t$ and $n$ times \textit{G$\hat{a}$teaux} differentiable with respect to $ X $. In addition, all their partial derivatives are continuous and bounded in $[0,T] \times H$. Furthermore, trajectories starting at $X \in E$ over the time horizon  $[t, T ]$  are denoted  $\Xb^{T}_{t,X} \equiv \Xb(T,t,\omega;X)$. Using this notation, we have that $\Xb(t,t,\omega;X) = X$.   Finally, for real separable Hilbert space $E$, by the notation $x \otimes y$ we mean a linear bounded operator on $E$ such that:
\begin{equation*}
    (x \otimes y)z = x  \langle y,z  \rangle, \quad \forall \, x,y,z \in E.
\end{equation*}

%In addition we  use the notation   $  \bbX^{'} \equiv  \bbX^{t_f}_{t+\rd t,X(t+\rd t)}  $  to denote stochastic trajectories starting from state $ X(t+ \rd t) $ at time instant $ t+\rd t $, while we use  $X^{'} \equiv X(t+\rd t) $ to represent the state at time instant $ t + \rd t $.

First, we perform the exponential transformation on the function  $ \psi\big(t,X(t)\big) \in C_{b}^{1,2}([0,T] \times H)$ and show that the transformed function $V\big(t,X(t)\big) \in C_{b}^{1,2}([0,T] \times H)$ satisfies the HJB equation for the case of infinite dimensional systems \cite{fabbri}. This result is derived with general $Q$-Wiener noise with covariance operator $Q$, however it holds also for cylindrical Wiener noise ($Q=I$). This will require  applying the Feynman-Kac lemma and deriving the backward Chapman Kolmogorov equation for the case  of infinite-dimensional stochastic  systems.  The backward Kolmogorov equations will result in the HJB equation after a logarithmic transformation is applied.   We start from the  free energy and relative entropy inequality in  \eqref{supeq:Legendre} and define the function $ \psi\big(t,X(t)\big):[0,T] \times H \rightarrow \mathbb{R}$  as follows: 
 \begin{equation*}
 \psi\big(t,X(t)\big) := \Eb_{\calL} \bigg[ \exp\big(  -\rho {J \big( \Xb^{T}_{t,X} )  } \big) \bigg| X\bigg], 
 \end{equation*}
which is simply the free energy as defined in \cref{def:Free_Energy}. By using the Feynman-Kac lemma we have that the function $   \psi(t, X) $ satisfies the backward Chapman Kolmogorov equation specified as follows:  
 \begin{align}
 \begin{split} \label{eq:Inf_Kolmogorov}
  - \partial_{t} \psi\big(t,X(t)\big) &= -\rho \ell\big(t, X(t)\big)  \psi\big(t,X(t)\big) + \big\langle \psi_{X}, \A X(t) + F\big(X(t)\big) \big\rangle + \frac{1}{2} \tr\bigg[\psi_{XX}(G Q^{\frac{1}{2}})  (G Q^{\frac{1}{2}})^{*}  \bigg].
 \end{split}
 \end{align} 
\noindent where $\partial_t \psi\big(t,X(t)\big)$ denotes the Fr\'{e}chet derivative of $\psi\big(t,X(t)\big)$ with respect to $t$, and $\psi_X$ and $\psi_{XX}$ denote the first and second G$\hat{\mbox{a}}$teaux derivatives of $\psi\big(t,X(t)\big)$ with respect to $X(t)$. Starting with the exponential transformation we have: 
 \begin{equation*}
  V\big(t,X(t)\big) = - \frac{1}{\rho} \log_{e} \psi\big(t,X(t)\big) \implies  \psi\big(t,X(t)\big) = e^{-\rho V(t,X(t))}. 
  \end{equation*}
\noindent Next we compute the functional derivatives $V_{X}$ and $V_{XX}$ as functions of the functional derivatives $\psi_{X}$ and $\psi_{XX}$. This results in:    
     \begin{align*}\begin{split} 
      \rho  \partial_{t}V\big(t,X(t)\big) e^{-\rho V}    & =  -\rho \ell\big(t, X(t)\big) e^{-\rho V}   - \rho  \big\langle V_{X} e^{-\rho V}, \A X(t) + F\big(X(t)\big) \big\rangle \\
      &\quad +  \frac{\rho}{2} \tr\bigg[ (V_{X} \otimes V_{X} ) (G Q^{\frac{1}{2}})  ( G Q{}^{\frac{1}{2}})^{*} e^{-\rho V}  \bigg] - \frac{1}{2} \tr\bigg[ (V_{XX}  (G Q^{\frac{1}{2}}) (G Q{}^{\frac{1}{2}})^{*} e^{-\rho V}  \bigg]. 
  \end{split}
\end{align*}   
 The last equations simplifies to:  
 \begin{align}\begin{split}\label{HJB0}
  - \partial_{t} V\big(t,X(t)\big) &= \ell\big(t, X(t)\big)  + \big\langle V_{X} , \A X(t) + F\big(X(t)\big) \big\rangle -  \frac{1}{2\rho}  \tr\bigg[ (V_{X} \otimes V_{X} ) (G Q^{\frac{1}{2}})  ( G Q{}^{\frac{1}{2}})^{*}  \bigg] \\
  &\quad +  \frac{1}{2 \rho }\tr\bigg[ V_{XX}  (G Q^{\frac{1}{2}}) (G Q{}^{\frac{1}{2}})^{*}   \bigg] 
  \end{split}
  \end{align}
  From the definition of the trace operator $\tr [A] := \sum_{j=1}^\infty \langle A e_j, e_j \rangle$ for  orthonormal basis $\lbrace e_j \rbrace$ over the domain of $A$, we have the following expression: 
\begin{align*}
 \frac{1}{2} \tr\bigg[ (V_{X} \otimes V_{X} ) (G Q^{\frac{1}{2\rho}})  ( G Q{}^{\frac{1}{2}})^{*}  \bigg]  = \frac{1}{2\rho}  \sum_{j=1}^{\infty} \big\langle ( V_{X} \otimes V_{X}) (G Q^{\frac{1}{2}})  ( G Q{}^{\frac{1}{2}})^{*} e_{j},e_{j} \big\rangle
\end{align*}  
\noindent  Since  $(x  \otimes y) z = x \langle y,z \rangle$ we have that:
  \begin{align*}
 \frac{1}{2\rho} \sum_{j=1}^{\infty} \big\langle ( V_{X} \otimes V_{X}) (G Q^{\frac{1}{2}})  ( G Q{}^{\frac{1}{2}})^{*} e_{j},e_{j} \big\rangle  &= \frac{1}{2\rho}  \sum_{j=1}^{\infty}   \Big\langle   V_{X}  \big\langle  V_{X} ,(G Q^{\frac{1}{2}})  ( G Q{}^{\frac{1}{2}})^{*} e_{j} \big\rangle ,e_{j} \Big\rangle  \\
  & = \frac{1}{2\rho}  \sum_{j=1}^{\infty}  \big\langle  V_{X} ,(G Q^{\frac{1}{2}})  ( G Q{}^{\frac{1}{2}})^{*} e_{j} \big\rangle   \big\langle   V_{X}  ,e_{j} \big\rangle  \\
  & = \frac{1}{2\rho}  \sum_{j=1}^{\infty}  \big\langle (G Q^{\frac{1}{2}})  ( G  Q{}^{\frac{1}{2}})^{*} V_{X},  e_{j} \big\rangle   \big\langle   V_{X}  ,e_{j} \big\rangle \\
  &\!\!\!\!\!\!\!\underset{\text{Parseval}} {=} \frac{1}{2} \big\langle   V_{X}, (G Q^{\frac{1}{2}})  ( G Q{}^{\frac{1}{2}})^{*} V_{X}   \big\rangle \\
  & = \frac{1}{2\rho}  \big|\big|  ( G Q{}^{\frac{1}{2}})^{*} V_{X}  \big|\big|_{U_{0}}^2 
\end{align*}
  Substituting back to \eqref{HJB0} we  have the HJB equation for the infinite dimensional case:
      \begin{align*}\begin{split} 
    -    V_{t}\big(t,X(t)\big)  &=  \ell\big(t, X(t)\big)   +     \big\langle V_{X} , \A X(t) +  F\big(X(t)\big) \big\rangle 
  +   \frac{1}{2 \rho} \tr\bigg[ V_{XX}  (G Q^{\frac{1}{2}}) (G Q{}^{\frac{1}{2}})^{*}   \bigg]  -   \frac{1}{2\rho}   ||  ( G Q{}^{\frac{1}{2\rho}})^{*} V_{X}  ||_{U_{0}}^2 
  \end{split}
\end{align*}

In the same vein, one can also show that the relative entropy between the probability measures induced by the uncontrolled and controlled infinite dimensional systems in eqs. (2) and (3) of the main text, respectively, results in an  infinite dimensional quadratic control cost. This requires the use of the Radon-Nikodym derivative from our generalization of Girsanov's theorem for the case of infinite dimensional stochastic systems in eqs. (2) and (3) of the main text.

%In cases where very high approximation accuracy (very low error) of the SPDE is required, a dense grid must be used for the spatial discretization (say a square grid with $l$ nodes per side). In this case, $n = l^2$ can be extremely large, yet a reasonable approximation of the noise yields $n>m$, producing a matrix $\mathcal{R}\in \mathbb{R}^{n \times m}$ that has fewer rows than columns, so that $\mathcal{R}\mathcal{R}^\top$ is rank deficient and therefore not invertible. The resulting derivation cannot be completed.

%The above discussion demonstrates that the infinite dimensional derivation is \textit{more general} than the equivalent derivation in finite dimensions. However, in cases where the expansion of the noise is large enough (i.e. $m \geq n$), the infinite dimensional derivation exactly follows the derivation in finite dimensions. This can be found in \cite{Grady_MPPI2015,Grady_MPPI2016,Grady_MPPI2018,Grady_ICRA_17,TRO_GRady_2018}, where the authors derive the above framework in finite dimensions for a variety of systems.

%\textbf{talk about Kappen?}

%%%%%%%%%%%%%%%%%%%%%%%%%%%%%%%%%%%%%%%%%%%%%%%%%%%%%%%%%%%%%%%%%%%%%%%%%%%%%%%%%%%%%%%%%%%%%%%%%%NEW_SECTION%%%%%%%%%%%%%%%%%%%%%%%%%%%%%%%%%%%%%%%%%%%%%%%%%%%%%%%%%%%%%%%%%%%%%%%%%%%%%%%%%%%%%%%%%%%%%%%%%%%%%%%%%%%

\subsection{SPDEs under Boundary Control and Noise}\label{supsec:BoundarySPDEs}
Let us consider the following problem with Neumann boundary conditions:
%\begin{equation}
%    \label{eq::laplace_dir}
%    \begin{cases}
%               \Delta_{\mathsf x} y(\mathsf x)=0,\quad \mathsf x\in \mathcal{O}\\
%               y(\mathsf x)=\gamma(\mathsf x), \quad \mathsf x\in \partial\mathcal{O}
%            \end{cases}
%\end{equation}
\begin{equation}
    \label{eq::laplace_neum}
    \begin{cases}
              \Delta_{\mathsf x} y(\mathsf x)=\lambda y(\mathsf x),\quad \mathsf x\in \mathcal{O}\\
              \frac{\partial}{\partial n}y(\mathsf x)=\gamma(\mathsf x), \quad 
              \mathsf x\in \partial\mathcal{O}
            \end{cases}
\end{equation}
where $\Delta_{\mathsf x}$ corresponds to the Laplacian, $\lambda\geq0$ is a real number, $\mathcal{O}$ is a bounded domain in $\mathbb{R}^{d}$ with regular boundary $\partial\mathcal{O}$ and $\frac{\partial}{\partial n}$ denotes the normal derivative, with $n$ being the outward unit normal vector. As shown in \cite{fabbri} and references therein, there exists a continuous operator %$\mathsf{D}_{D}:H^{s}(\partial\mathcal{O})\rightarrow H^{s+1/2}(\mathcal{O})$ and
$\mathsf{D}_{N}:H^{s}(\partial\mathcal{O})\rightarrow H^{s+3/2}(\mathcal{O})$ such that %$\mathsf{D}_{D}\gamma$ and 
$\mathsf{D}_{N}\gamma$ is the solution to \eqref{eq::laplace_neum}. Given this operator, stochastic parabolic equations with Neumann boundary conditions of the following type:
\begin{equation}
\label{eq::bcp}
\begin{split}
&\frac{\partial h(t, \mathsf x)}{\partial t}=\Delta_{\mathsf x}h(t, \mathsf x)+f_1(t, h)+c_1(t, h)\frac{\partial w(t,\mathsf x)}{\partial t},\quad \mathsf x\in \mathcal{O}\\
%&\begin{cases}
&\frac{\partial h(t, \mathsf x)}{\partial n}=f_2(t, h)+c_2(t, h)\frac{\partial v(t,\mathsf x)}{\partial t} ,\quad \mathsf x\in \mathcal{\partial O},\\
%h(t, \mathsf x)&=f_2(t, h)+c_2(t, h)\frac{\partial v(t,\mathsf x)}{\partial t},\quad \mathsf x\in \mathcal{\partial O},\quad \text{for Dirichlet boundary conditions,}
%\\
%\end{cases}\\
&h(0, \mathsf x) = h_0(\mathsf x).
\end{split}
\end{equation}
can be written in the mild abstract form:
\begin{equation}
    \label{eq::abs_bound}
    \begin{split}
        \begin{cases}
        \begin{split}
        X(t&)=e^{t\mathcal{A}_N}X_{0}+\int_{0}^{t}e^{(t-s)\mathcal{A}_N}F_1(s, X)\rd s+\int_{0}^{t}e^{(t-s)\mathcal{A}_N}C_1(s, X)\rd W(s)\\
        &\quad +\int_{0}^{t}(\lambda I-\mathcal{A}_N)^{1/4+\epsilon}e^{(t-s)\mathcal{A}_N}G_NF_2(s, X)\rd s+\int_{0}^{t}(\lambda I-\mathcal{A}_N)^{1/4+\epsilon}e^{(t-s)\mathcal{A}_N}G_NC_2(s, X)\rd V(s),
        \end{split}
        \end{cases}
    \end{split}
\end{equation}
where $G_N:=(\lambda I-\mathcal{A}_N)^{3/4-\epsilon}\mathsf{D}_N$, and the remaining terms are defined with respect to the space-time formulation of \eqref{eq::abs_bound}. A similar expression can be obtained for Dirichlet conditions as well, however the solution has to be investigated under weak norms, or in weighted $L^2$ spaces. More details can be found in \cite[Appendix C]{fabbri} and references therein.

%%%%%%%%%%%%%%%%%%%%%%%%%%%%%%%%%%%%%%%%%%%%%%%%%%%%%%%%%%%%%%%%%%%%%%%%%%%%%%%%%%%%%%%%%%%%%%%%%%NEW_SECTION%%%%%%%%%%%%%%%%%%%%%%%%%%%%%%%%%%%%%%%%%%%%%%%%%%%%%%%%%%%%%%%%%%%%%%%%%%%%%%%%%%%%%%%%%%%%%%%%%%%%%%%%%%%

\subsection{An Equivalence of the Variational Optimization approach for SPDEs with Q-Wiener Noise}\label{supsec:Q_Wiener_derivation}

In this section we briefly discuss how one obtains an equivalent variational optimization as in Section \rm{III} of the main text, for control of SPDEs with $Q$-Wiener noise. Consider the uncontrolled and controlled version of an $H$-valued process be given, respectively, by:

\begin{align}
\rd X &= \big(\A X   + F(t, X)\big)  \rd t +   \frac{1}{\sqrt{\rho}} \sqrt{Q} \rd W(t), \label{supeq:Q_Wiener_SPDEs_NoControl}\\
\rd \tilde{X} &= \big( \A \tilde{X}   + F(t, \tilde{X})  \big) \rd t  + \sqrt{Q}\big(\calU(t, \tilde X)\rd t+ \frac{1}{\sqrt{\rho}} \rd  W(t)\big), \label{supeq:Q_Wiener_SPDEs_Control}
\end{align}
with initial condition $X(0) = \tilde{X}(0) = \xi$. Here, $Q$ is a trace-class operator, and $W \in U$ is a cylindrical Wiener process. The assumption that $Q$ is of trace class is expressed as:
\begin{equation*}
    \text{Tr}\big[Q\big] = \sum_{n=1}^\infty \big\langle Q e_n, e_n \big \rangle < \infty.
\end{equation*}

As opposed to the discussion following \cref{eq:SPDEs_Control} of the main text, in this case we do not require any contractive assumption on the operator $\A$ due to the nuclear property of the operator $Q$. The stochastic integral $\int_{0}^{t}e^{(t-s)\A}\sqrt{Q}\rd W(s)$ is well defined in this case \cite[Chapter 4.2]{da1992stochastic}. Define the process:
\begin{equation*}
\begin{split}
    W_Q(t) &:= \sqrt{Q}W(t) = \sum_{n=1}^\infty \sqrt{Q}e_n \beta_n(t) \\
    &= \sum_{n=1}^\infty \sqrt{\lambda_n} e_n \beta_n(t)
\end{split}
\end{equation*}
where the basis $\lbrace e_n \rbrace$ satisfies the eigenvalue-eigenvector relationship $Qe_n = \lambda e_n$. The process $W_Q(t)$ satisfies the properties in Definition \ref{def:q_wiener}, and is therefore a $Q$-Wiener process. 

The above case is an SPDE driven by Q-Wiener noise, which is quite different from the cylindrical Wiener process described in the rest of this work. In order to state the Girsanov's theorem in this case, we first define the Hilbert space $U_0 := \sqrt{Q}(U) \subset U$ with inner product $\langle u, v \rangle_{U_0} := \big\langle Q^{-1/2}u, Q^{-1/2} v \big\rangle_U$, $\forall u,v \in U_0$.

\begin{theorem}[Girsanov] \label{sup:girs} Let $\Omega$ be a sample space with a $\sigma$-algebra $\mathcal{F}$. Consider the following $H$-valued stochastic processes:
\begin{align}
\rd X &= \big(\A X   + F(t, X)\big)  \rd t +   \frac{1}{\sqrt{\rho}} \rd W_Q(t), \label{sup:X}\\
\rd \tilde{X} &= \big( \A \tilde{X}   + F(t, \tilde{X})  \big) \rd t  + \sqrt{Q}\calU(t, \tilde X)\rd t+ \frac{1}{\sqrt{\rho}} \rd  W_Q(t)\big), \label{sup:X_tilde}
\end{align}
where $X(0)=\tilde{X}(0)=x$ and $W_Q\in U$ is a Q-Wiener  process with respect to measure $\mathbb{P}$. Moreover, for each $\Gamma\in C([0,T]; H)$, let the {\it law} of $X$ be defined as $\mathcal{L}(\Gamma):=\mathbb{P}(\omega\in\Omega|X(\cdot,\omega)\in\Gamma)$. Similarly, the law of $\tilde{X}$ is defined as $\tilde{\calL}(\Gamma):=\mathbb{P}(\omega\in\Omega|\tilde{X}(\cdot,\omega)\in\Gamma)$. Then
\begin{equation} \label{supeq:Lg}
\begin{split}
\tilde{\calL}(\Gamma)= \textstyle\mathbb{E}_{\mathbb{P}}\big[\exp\big(\int_{0}^{T}\big\langle\psi(s), dW_Q(s)\big\rangle_{U_0}-\frac{1}{2}\int_{0}^{T}||\psi(s)||_{U_0}^{2}ds\big)|X(\cdot)\in\Gamma\big],
\end{split}
\end{equation}
where we have defined $\psi(t):=\sqrt{\rho}\calU\big(t, \tilde{X}(t)\big)\in U_{0}$ and assumed
\begin{equation}\label{sup:psi_assum1}
    \mathbb{E}_{\mathbb{P}}\Big[e^{\frac{1}{2}\int_{0}^{T}||\psi(t)||^2\mathrm dt}\Big]<+\infty.
\end{equation}% Here, we write for brevity $\tilde{\calL}(\omega)\equiv\tilde{\calL}(\tilde{X}(\cdot,\omega)\in\Gamma)$.
\end{theorem}
\begin{proof}
The proof is identical to the proof of Theorem \ref{girs}.
\end{proof}

Note that $\psi(t)$ in this case is identical to $\psi(t)$ in Theorem \ref{girs}. As a result, despite having $Q$-Wiener noise,  we have the same variational optimization for this case as in Section \rm{III} of the main text.
%\textbf{why have this result here? also this is only true when $Q$ multiplies controls as well}

%%%%%%%%%%%%%%%%%%%%%%%%%%%%%%%%%%%%%%%%%%%%%%%%%%%%%%%%%%%%%%%%%%%%%%%%%%%%%%%%%%%%%%%%%%%%%%%%%%NEW_SECTION%%%%%%%%%%%%%%%%%%%%%%%%%%%%%%%%%%%%%%%%%%%%%%%%%%%%%%%%%%%%%%%%%%%%%%%%%%%%%%%%%%%%%%%%%%%%%%%%%%%%%%%%%%%

\subsection{A Comparison to Variational Optimization in Finite Dimensions}\label{supsec:Comparison_finite}

In what follows we show how degeneracies arise for a similar derivation in finite dimensions. The stochastic dynamics are given by:
\begin{equation}\label{supeq:SPDEs_control}
  \rd X = \big(\A X   + F(t, X) \big) \rd t+G(t, X) \big(\calU(t, X)  \rd t +   \frac{1}{\sqrt{\rho}}  \rd W(t) \big),
 \end{equation} 
where W(t) is a cylindrical Wiener process. Now, let the Hilbert space state vector $X(t)\in H$ be approximated by a finite dimensional state vector $X(t) \approx \hat{X}(t)\in \Rb^n$ with arbitrary accuracy, where $n$ is the number of grid points. In order to rewrite a finite dimensional form of \eqref{supeq:SPDEs_control}, the cylindrical Wiener noise term $W(t)$ must be captured by a finite dimensional approximation. The expansion of $W(t)$ in \eqref{eq:Wiener_expansion} is restated here and truncated at $m$ terms:
 \begin{equation}\label{supeq:Q_Wiener_approx}
    W(t) = \sum_{j=1}^{\infty}  \sqrt{\lambda_{j}} \beta_{j}(t) e_{j} =  \sum_{j=1}^{\infty}  \beta_{j}(t) e_{j} \approx \sum_{j=1}^m \beta_j(t) e_j %= \mathcal{R} \vbeta 
 \end{equation}
where $\lambda_{j} = 1$, $\forall j \in \mathbb{N}$ in the case of cylindrical Wiener noise, and $\beta_j(t)$ is a standard Wiener process on $\mathbb{R}$. The stochastic dynamics in \eqref{supeq:SPDEs_control} become a finite set of SDEs:
\begin{equation}\label{supeq:Finite_DimensionalSDE}
\rd \hat{X} = \big(\mathcal{A} \hat{X} + \mathcal{F}(t,\hat{X}) \big)\rd t + \mathcal{G}(t,\hat{X})\big( \mathcal{M} \vu(t;\vtheta) \rd t + \frac{1}{\sqrt{\rho}} \mathcal{R}\rd \vbeta(t) \big)
\end{equation}
 The terms   $\mathcal{A}$, $\mathcal{F}$, and $\mathcal{G}$ are matrices associated with the Hilbert space operators $\A$, $F$, and $G$ respectively. The matrix  $ \mathcal{M}  $   has dimensionality $\mathcal{M} \in \mathbb{R}^{n \times k} $, where $ k  $ is the number of actuators placed in the field.    The vector $\rd \vbeta \in \mathbb{R}^m$ collects the Wiener noise terms in the expansion \eqref{supeq:Q_Wiener_approx}, and the matrix $\mathcal{R}$ collects finite dimensional basis vectors from \eqref{supeq:Q_Wiener_approx}. As noted in the main paper, the dimensionality of the $\mathcal{R} $ is $\mathcal{R}  \in \mathbb{R}^{n \times m}  $. The degeneracy arises when $ n>m $ for the case of the cylindrical noise. For the case of Q-Wiener noise, degeneracy may arises even when $ n\leq m  $ and Rank$(\mathcal{R}) < n $.  In both cases, the issue of degeneracy   prohibits the use of Girsanov theorem for the importance sampling steps due to the lack of invertibility of  $ \mathcal{R} $.   With respect to the approach relying on Gaussian densities, the derivation would  require the following time discretization of the reduced order model in \eqref{supeq:Finite_DimensionalSDE}: 
  
  \begin{align}
     \hat{X}(t+ \Delta t)   &=   \hat{X}(t) + \int_{t}^{t+\Delta t} \Big(\mathcal{A} \hat{X} + \mathcal{F}(t,\hat{X}) \Big)\rd t +   \int_{t}^{t+\Delta t}\mathcal{G}(t,\hat{X})\Big( \mathcal{M} \vu(t;\vtheta) \rd t + \frac{1}{\sqrt{\rho}} \mathcal{R}\rd \vbeta(t) \Big)  \\
     & \approx    \hat{X}(t) +  \Big(\mathcal{A} \hat{X} + \mathcal{F}(t,\hat{X}) \Big)\Delta t +   \mathcal{G}(t,\hat{X})\Big( \mathcal{M} \vu(t;\vtheta) \Delta t + \frac{1}{\sqrt{\rho}} \mathcal{R}\rd \vbeta(t) \Big) \\ 
  \end{align}

Without loss of generality we  simplify the expression above by assuming the  $ \mathcal{G}(t,\hat{X}) = I_{n \times n}$. The transition probability will take the following form:
\small
\begin{equation}
 \rp\big(\hat{X}(t+ \Delta t) | \hat{X}(t)\big) = \frac{1}{(\sqrt{2 \pi})^{n} (\det{\Sigma_{\hat{X}}})^{\frac{1}{2}}} \exp\bigg(- \frac{1}{2}\Big(\hat{X}(t+ \Delta t) - \mu_{\hat{X}}(t+\Delta t)  \Big)^{\top} \Sigma_{\hat{X}}^{-1} \Big(\hat{X}(t+ \Delta t) - \mu_{\hat{X}}(t+ \Delta t)  \Big) \bigg)  
\end{equation}\normalsize

where the term   $ \mu_{\hat{X}}(t+ \Delta t) $  is the mean and  $\Sigma_{\hat{X}} $  is the variance defined as follows:

\begin{align}
    \mu_{\hat{X}}(t+ \Delta t) &=   \hat{X}(t) +  \big(\mathcal{A} \hat{X} + \mathcal{F}(t,\hat{X}) \big)\Delta t + \mathcal{M} \vu(t;\vtheta) \Delta t \\
    \Sigma_{\hat{X}} &=  \frac{1}{\rho} \mathcal{R}\mathcal{R}^{\rT} \Delta t
\end{align}
The existence of the transition probability densities requires invertibility of $ \mathcal{R} \mathcal{R}^{\rT} $ which is not possible when $ n<m $ or when $ \text{Rank}(\mathcal{R})<n $ for  $ n \geq m $. 

\section{Details on Numerical Implementation}\label{supsec:results}

\subsection{Algorithms for Open Loop and Model Predictive Infinite Dimensional Controllers}
\label{supsec:Pseudocode}

The following algorithms use equations derived in \cite{lord_powell_shardlow_2014} for finite difference approximation of semi-linear SPDEs for Dirichlet and Neumann Boundary conditions. Spatial discretization is done as follows: pick a number of coordinate-wise discretization points $J$ on the coordinate-wise domain $\mathcal{D}=[a,b] \subset \mathbb{R}$ such that each spatial coordinate is discretized as $x_k=a+k\,\frac{b-a}{J}$ where $k=0,\,1,\,2,\,\dots,J$. For our experiments, the function that specifies how actuation is implemented by the infinite dimensional control is of the following form: 
\begin{equation}
m_l (x_k;\theta)=\exp \Big[ \frac{-1}{2 \sigma_l^2} (x_k - \mu_l)^2 \Big], \quad l=1,\dots,N
\label{eqn:actuator}
\end{equation}
where, $\mu_l$ denotes the spatial position of the actuator on $[a,b]$ and $\sigma_l$ controls the influence of the actuator on nearby positions.
\par Next we provide two algorithms for infinite dimensional stochastic control. In particular, Algorithm \ref{Algorithm1} is for open-loop trajectory optimization and Algorithm \ref{Algorithm2} is for Model Predictive control that uses implicit feedback.

\begin{algorithm}[H]
 \caption{Open Loop Infinite Dimensional Controller}
 \begin{algorithmic}[1]
 \State \textbf{Function:} \textit{u = \textbf{OptimizeControl}(Time horizon ($T$), number of optimization iterations ($I$), number of trajectory samples per optimization iteration ($R$), initial field profile ($X_0$), number of actuators ($N$), initial control sequences ($u_{T\times N}$) for each actuator, temperature parameter ($\rho$), time discretization ($\Delta t$), actuator centers and variance parameters ($\theta$))}
 \For{$i=1 \;\text{to}\; I$}
 \State Initialize $X \gets X_0$
 \For{$r=1\;\text{to}\;R$}
 \For{$t=1\;\text{to}\;T$}
  \State Sample noise, $\rd W(t,x_k)=\sum_{j=1}^{J}\big( \,e_j(t,x_k)\,\beta_j(t)\big)$, $e_j=\sqrt{2/a} \,sin(j\pi x/a)$ for $x\in L^2(0,a)$  
  \State Compute entries of the actuation matrix $\tilde{M}$ by \eqref{eqn:actuator}
  \State Compute the control actions applied to each grid point, $\mathcal{U}(t)=u(t)^{T}\,\tilde{M}$
  \State Propagate the discretized field $X(t)$ \cite[Algorighm 10.8]{lord_powell_shardlow_2014}
  \Statex \quad \quad \quad \textbf{end for} 
 \EndFor
  \State Compute trajectory cost $J_r^{(i)}$ via \cref{eq:J_i} of the main text
  \Statex \quad \quad \textbf{end for}
  \EndFor
  \State Compute exponential weight of each trajectory $\mathcal{J}_r^{(i)} := \exp\big(-\rho J_r^{(i)}(X)\big)$
  \State Compute the normalizer $\mathcal{J}_m^{(i)}=\frac{1}{R}\,\sum_{r=1}^{R}\mathcal{J}_r^{(i)} $
  \State Update nominal control sequence by \cref{eq:Iterative_optimalvariation1} of the main text
  \Statex \textbf{end for} 
 \EndFor
 \State \textbf{Return:} u
 \end{algorithmic}
 \label{Algorithm1}
\end{algorithm}

\begin{algorithm}[H]
\caption{Model Predictive Infinite Dimensional Controller}
\begin{algorithmic}[1]
\State \textbf{Inputs:} MPC time horizon ($T$), number of optimization iterations ($I$), number of trajectory samples per optimization iteration ($R$), initial profile ($X_0$), number of actuators ($N$), initial control sequences ($u_{T\times N}$) for each actuator, temperature parameter ($\rho$), time discretization ($\Delta t$), actuator centers and variance parameters ($\theta$), total simulation time ($T_{\text{sim}}$)
\For{$t_{\text{sim}}=1 \;\text{to}\; T_{\text{sim}}$}
    \State $u_I(t_{\text{sim}}) = $ \textbf{\textit{OptimizeControl}} $\!\!(T,I,R,X_0, N, u, \rho, \Delta t, \theta)$ 
    \State Apply $u_I(t=1)$ and propagate the discretized field to $t_{\text{sim}}+1$
    \State Update the initial field profile $X_0 \gets X(t_{\text{sim}}+1)$
    \State Update initial control sequence $u= \big[ u_I[2:T,:];\;u_I[T,:]\big]$
    \Statex \textbf{end for}    
\EndFor
\end{algorithmic}
\label{Algorithm2}
\end{algorithm}

For MATLAB pseudo-code on sampling space-time noise (step 6 in algorithm \ref{Algorithm1} and step 7 in algorithm \ref{Algorithm2}), refer to \cite[algorithms 10.1 and 10.2]{lord_powell_shardlow_2014}. Note however, that our experiments used cylindrical Wiener noise so $\lambda_j = 1$ $\forall j=1,\dots,J$.

%%%%%%%%%%%%%%%%%%%%%%%%%%%%%%%%%%%%%%%%%%%%%%%%%%%%%%%%%%%%%%%%%%%%%%

\subsection{Brief description of each experiment}
The following is additional information about the experiments referenced in Section V. Section \ref{supsubsec:HeatSPDE} describes boundary and distributed control experiments, while Sections \ref{supsubsec:BurgersSPDE} and \ref{supsubsec:NagumoSPDE} describe experiments for distributed control only.
\subsubsection{Heat SPDE} \label{supsubsec:HeatSPDE}
The 2D stochastic Heat PDE with homogeneous Dirichlet boundary conditions given by: 
\begin{equation} \label{supeq:HeatSPDE}
\begin{split}
h_t(t, x, y) &= \epsilon h_{xx}(t,x,y) + \epsilon h_{yy}(t,x,y) + \sigma dW(t), \\
h(t,0,y) &= h(t,a,y) = h(t,x,0)= h(t,x,a)=0, \\ 
h(0,x,y) &\sim \mathcal{N}(h_0;0,\sigma_0), 
\end{split}
\end{equation}
where the parameter $\epsilon$ is the so called thermal diffusivity, which governs how quickly the initial temperature profile diffuses across the spatial domain. \eqref{supeq:HeatSPDE} considers the scenario of controlling a metallic plate to a desired temperature profile using 5 actuators distributed across the plate. The edges of the plate are always held at constant temperature of 0 degrees Celsius. The parameter $a$ is the length of the sides of the square plate, for which we use $a=0.5$ meters. 

The actuator dynamics are modeled by Gaussian-like exponential functions with the means co-located with the actuator locations at: $\vmu = \big[\mu_1, \mu_2, \mu_3, \mu_4, \mu_5 \big] =  \big[(0.2a,0.5a), (0.5a,0.2a), (0.5a,0.5a), \\ (0.5a,0.8a), (0.8a,0.5a)\big]$ and the variance of the effect of each actuator on nearby field states given by $\sigma_l^2 = (0.1a)^2$, $\forall l = 1, \dots, 5$. For every $j= 1, \dots, J$, and $l=1, \dots, N$, the resulting $m_l(\vx)$ has the form:
\small
\begin{equation*}
    m_{l,j}\left(\left[\begin{array}{c} x \\ y\end{array} \right]\right) = \exp \left\lbrace -\frac{1}{2}\left(\left[ \begin{array}{c} x \\ y
    \end{array} \right] - \left[ \begin{array}{c} \mu_{l,x} \\ \mu_{l,y} \end{array} \right] \right)^\top \left[ \begin{array}{cc}
         \sigma_l^2 & 0 \\
         0 & \sigma_l^2 
    \end{array} \right] \left(\left[ \begin{array}{c} x \\ y
    \end{array} \right] - \left[ \begin{array}{c} \mu_{l,x} \\ \mu_{l,y} \end{array} \right] \right) \right\rbrace
\end{equation*}\normalsize

The spatial domain is discretized by dividing the x and y domains into 64 points each creating a grid of $64 \times 64$ spatial locations on the plate surface. For our experiments, we use a semi-implicit forward Euler discretization scheme for time and central difference for the $2^{nd}$ order spatial derivatives $h_{xx}$ and $h_{yy}$. We used the following parameter values, time discretization $\Delta t=0.01s$, MPC time horizon $T = 0.05s$, total simulation time $T_{\text{sim}} = 1.0s$, thermal diffusivity $\epsilon=1.0$ and initialization standard deviation $\sigma_0=0.5$.
The cost function considered for the experiments was defined as follows: 
\begin{equation*}
    J := \sum_{t} \sum_{x} \sum_{y} \;\kappa \big(h_{\text{actual}}(t,x,y) - h_{\text{desired}} (t,x,y)\big)^2 \cdot \mathbbm{1}_{S}(x,y)
\end{equation*}
where $S := \cup_{i=1}^5 S_i$ and the indicator function $ \mathbbm{1}_{S}(x,y) $  is defined as follows:
\begin{equation}\label{supeq:indicator}
 \mathbbm{1}_S(x,y) :=
\begin{cases}
1,  \quad \text{if} ~~~(x,y) \in S  \\
0, \quad  \text{otherwise}
\end{cases}
\end{equation}
% \quad  \text{if}  ~~~0.48a  \leq x \leq 0.52 \;\text{and}\; 0.48a  \leq y \leq 0.52,   ~~~  \text{or} ~~~0.48  \leq x \leq 0.52   ~~~ \text{or} ~~~  0.78  \leq x \leq 0.82
where\\ 
$S_1 = \{(x,y) \mid  x \in [0.48a, 0.52a] \;\text{and}\; y \in [0.48a, 0.52a]\}$ is in the central region of the plate\\
$S_2 = \{(x,y) \mid  x \in [0.22a, 0.18a] \;\text{and}\; y \in [0.48a, 0.52a]\}$ is the left-mid region of the plate \\
$S_3 = \{(x,y) \mid x \in [0.82a, 0.78a] \;\text{and}\; y \in [0.48a, 0.52a]\}$ is the right-mid region of the plate \\
$S_4 = \{(x,y) \mid x \in [0.48a, 0.52a] \;\text{and}\; y \in [0.18a, 0.22a]\}$ is in the top-central region of the plate \\
$S_5 = \{(x,y) \mid x \in [0.48a, 0.52a] \;\text{and}\; y \in [0.78a, 0.82a]\}$ is in the bottom-central region of the plate \\
\\
In addition  $ h_{\text{desired}} (t,x,y)= 0.5^{\circ}\,C $  for $(x,y) \in S_1$ and   $ h_{\text{desired}} (t,x,y)= 1.0^{\circ}\,C $  for $(x,y) \in \cup_{i=2}^5 S_i$ and the scaling parameter $\kappa=100$.

In the boundary control case, we make use of the 1D stochastic heat equation given as follows: 
\begin{align*}
h_t(t,x) &= \epsilon h_{xx}(t,x) +  \sigma dW(t)\\
h(0,x) &= h_0(x)
\end{align*}

For Dirichlet and Neumann boundary conditions we have $h(t,x) =\gamma(x)$, $\forall x\in\partial O$ and $h_x(t,x) =\gamma(x)$, $\forall x\in\partial O $, respectively. Regarding our 1-D boundary control example, we set $\epsilon=1 $, $\sigma=0.1$, $h_x(t,0)=u_1(t)$ and $h_x(t,a)=u_2(t)$. In this case, $m_l(x)$ is simply given by the identity function and the corresponding inner products associated with Girsanov's theorem are given by the standard dot product. Finally, the cost function used is the same as above with $S=\{x|0<x<a\}$ and
\begin{equation*}
h_{desired}(t,x)=
\begin{cases}
1, \quad\text{for}\hspace{1mm} t\in[0, 0.4],\\
3, \quad\text{for}\hspace{1mm} t\in[0, 0.4]\hspace{1mm} \text{and}\hspace{1mm} t\in[0.8, 1.3].
\end{cases}
\end{equation*}

\subsubsection{Burgers SPDE} \label{supsubsec:BurgersSPDE}
The 1D stochastic Burgers PDE with non-homogeneous Dirichlet boundary conditions is as follows:
\begin{equation} \label{supeq:BurgersSPDE}
\begin{split}
h_t(t, x) + h h_x(t, x) &= \epsilon h_{xx}(t,x) + \sigma dW(t)\\
h(t,0) &= h(t,a) = 1.0\\ 
h(0,x) &= 0, \; \forall x \in (0,a)
\end{split}
\end{equation}
where the parameter $\epsilon$ is the viscosity of the medium. \eqref{supeq:BurgersSPDE} considers a simple model of a 1D flow of a fluid in a medium with non-zero flow velocities at the two boundaries. The goal is to achieve and maintain a desired flow velocity profile at certain points along the spatial domain. As seen in the desired profile in Fig. 3 of the main paper, there are 3 areas along the spatial domain with desired flow velocity such that the flow has to be accelerated, then decelerated, and then accelerated again while trying to overcome the stochastic forces and the dynamics governed by the Burgers PDE. Similar to the experiments for the Heat SPDE, we consider actuators behaving as Gaussian-like exponential functions with the means co-located with the actuator locations at: $\vmu = \big[0.2a, 0.3a, 0.5a, 0.7a, 0.8a\big]$ and the spatial effect (variance) of each actuator given by $\sigma_l^2 = (0.1a)^2$, $\forall \, l = 1, \dots, 5$. The parameter $a=2.0\;m$ is the length of the channel along which the fluid is flowing. 

This spatial domain was discretized using a grid of 128 points. The numerical scheme used semi-implicit forward Euler discretization for time and central difference approximation for both the $1^{st}$ and $2^{nd}$ order derivatives in space. The $1^{st}$ order derivative terms in the advection term $h h_x$ were evaluated at the current time instant while the $2^{nd}$ order spatial derivatives in the diffusion term $h_{xx}$ were evaluated at the next time instant, hence the scheme is semi-implicit. Following are values of some other parameters used in our experiments: time discretization $\Delta t=0.01$, total simulation time = $1.0\,s$, MPC time horizon = $0.1\,s$, and the scaling parameter $\kappa=100$. The cost function considered for the experiments was defined as follows: 
\begin{equation*}
    J := \sum_{t} \sum_{x} \;\kappa \big(h_{\text{actual}}(t,x) - h_{\text{desired}} (t,x)\big)^2 \cdot \mathbbm{1}_S(x)
\end{equation*}
where the function  $ \mathbbm{1}_S(x) $  is defined as in \eqref{supeq:indicator} with $S = \cup_{i=1}^3$, where $S_1 = [0.18a,0.22a]$, $S_2 = [0.48a,0.52a]$, and $S_3 = [0.78a,0.82a] $.
%follows 
%$$
% \mathbbm{1}(x) =
%\begin{cases}
%1, \quad  \text{if }  x \in 
%\lbrace[0.18a,0.22a],[0.48a,0.52a],[0.78a,0.82a]\rbrace \\%~~~0.18a  \leq x \leq 0.22a,  ~~~  \text{or} ~~~0.48a  \leq x \leq 0.52a   ~~~ \text{or} ~~~  0.78a  \leq x \leq 0.82a \\
%0, \quad  \text{otherwise}
%\end{cases}
%$$
In addition  $ h_{\text{desired}} (t,x)= 2.0 \;m/s$  for $x \in S_1 \cup S_3$ which is at the sides, and $ h_{\text{desired}} (t,x)= 1.0 \;m/s$  for $x\in S_2$ which is in the central region.

\subsubsection{Nagumo SPDE} \label{supsubsec:NagumoSPDE} The stochastic Nagumo equation with Neumann boundary conditions is as follows: 
\begin{align*}
h_t(t,x) &= \epsilon h_{xx}(t,x) + h(t,x)\big(1-h(t,x)\big)\big(h(t,x)-\alpha\big) + \sigma dW(t)\\ h_x(t,0) &= h_x(t,a) = 0\\ 
h(0,x) &= \big(1+\exp(-\frac{2-x}{\sqrt[]{2}})\big)^{-1}
\end{align*}
The parameter $\alpha$ determines the speed of a wave traveling down the length of the axon and $\epsilon$ the rate of diffusion. By simulating the deterministic Nagumo equation with $a=5.0,\,\epsilon=1.0$ and $\alpha=-0.5$, we observed that after about 5 seconds, the wave completely propagates to the end of the axon. Similar to the experiments for the Heat SPDE, we consider actuators behaving as Gaussian-like exponential functions with actuator centers (mean values) at $\vmu = \big[0.2a, 0.3a, 0.4a, 0.5a, 0.6a, 0.7a, 0.8a\big]$ and the spatial effect (variance) of each actuator given by $\sigma_l^2 = (0.1a)^2$, $\forall \, l = 1, \dots, 7$. The spatial domain was discretized using a grid of 128 points. The numerical scheme used semi-implicit forward Euler discretization for time and central difference approximation for the $2^{nd}$ order derivatives in space. Following are values of some other parameters used in our experiments: time discretization $\Delta t=0.01$, MPC time horizon = $0.1\,s$, total simulation time = $1.5\,s$ for acceleration task and total simulation time = $5.0\,s$ for the suppression task, and the scaling parameter $\kappa=10000$. The cost function for this experiment was defined as follows: 
\begin{equation*}
    J = \sum_{t} \sum_{x} \;\kappa \big(h_{\text{actual}}(t,x) - h_{\text{desired}} (t,x)\big)^2 \cdot \mathbbm{1}_S(x)
\end{equation*}
where $ h_{\text{desired}} (t,x)= 0.0\; V$ for the suppression task, and $ h_{\text{desired}} (t,x)= 1.0\; V $ for the acceleration task, and the function  $\mathbbm{1}_S(x) $  is defined as in \eqref{supeq:indicator} with $S=[0.7a,0.99a]$.

\end{document}

%% file: Packages.tex
\usepackage{setspace}
%\doublespacing

\usepackage[utf8]{inputenc} % allow utf-8 input
\usepackage[T1]{fontenc}    % use 8-bit T1 fonts
\usepackage{url}            % simple URL typesetting
\usepackage{booktabs}       % professional-quality tables
\usepackage{amsfonts}       % blackboard math symbols
\usepackage{nicefrac}       % compact symbols for 1/2, etc.
\usepackage{microtype}      % microtypography
\usepackage{amsthm}
\usepackage{subfigure}

\usepackage{capt-of}
% \usepackage{subcaption}
% \captionsetup{justification=raggedright,singlelinecheck=false}
% \usepackage{subcaption} % does not work with rev-tex
% \usepackage{graphics} % for pdf, bitmapped graphics files
\usepackage{epsfig} % for postscript graphics files
\usepackage{times} % assumes new font selection scheme installed
\usepackage{amsmath} % assumes amsmath package installed
\usepackage{amssymb}  % assumes amsmath package installed
\usepackage{makeidx}
\usepackage{float}
\usepackage{balance}
\usepackage{booktabs}
\usepackage{bbm}
\usepackage{mathrsfs}
\usepackage{enumitem}
\usepackage{algorithm}
\usepackage{algpseudocode}

\usepackage[super]{nth} %for writing 2nd, 3rd, etc nicely

\usepackage{mathrsfs}  
\usepackage[normalem]{ulem}
\usepackage{xr} % enable referencing labels in other documents
\usepackage{xcolor}
\usepackage{float}
% \usepackage[nolist]{acronym}

%  This is used for custom colored hyperlinks within the document to point to equations or references. If submitting to a journal or conference, use their defaults and comment this out
\usepackage{xcolor}
\usepackage[colorlinks]{hyperref}       % hyperlinks

\hypersetup{
  colorlinks = true,
  allcolors = siaminlinkcolor,
  urlcolor = siamexlinkcolor,
}
\colorlet{siaminlinkcolor}{blue!50!black}
\colorlet{siamexlinkcolor}{blue!50!black}

\usepackage{lipsum} %this package can be removed

\usepackage{setspace}
%\doublespacing

\usepackage{cleveref} %Cleveref must be defined after hyperref

\usepackage[acronym]{glossaries}

\usepackage{etoolbox}

%% file: Commands.tex
\newcommand{\rd}{{\mathrm d}}

\newcommand{\rp}{{\mathrm p}}

\newcommand{\rT}{{\mathrm{T}}}

\newcommand{\vx}{{\bf x}}

\newcommand{\vu}{{\bf u}}

% Capital boldface letters

\newcommand{\vM}{{\bf M}}

%Bold face Greek letters

\newcommand{\vmu}{{\mbox{\boldmath$\mu$}}}
\newcommand{\vbeta}{{\mbox{\boldmath$\beta$}}}

\newcommand{\vtheta}{{\mbox{\boldmath$\theta$}}}

%Capital caligraphic letters

\newcommand{\calB}{{\cal B}}

\newcommand{\calF}{{\cal F}}

\newcommand{\calH}{{\cal H}}

\newcommand{\calL}{{\cal L}}
\newcommand{\calM}{{\cal M}}
\newcommand{\calN}{{\cal N}}

\newcommand{\calR}{{\cal R}}

\newcommand{\calU}{{\cal U}}

%Bold face caligraphic letters

\newcommand{\argmax}{\operatornamewithlimits{argmax}}
\newcommand{\argmin}{\operatornamewithlimits{argmin}}

\newcommand{\Eb}{\mathbb{E}}

\newcommand{\Nb}{\mathbb{N}}
\newcommand{\Pb}{\mathbb{P}}
\newcommand{\Qb}{\mathbb{Q}}
\newcommand{\Rb}{\mathbb{R}}

\newcommand{\Xb}{\mathbb{X}}

\newcommand{\tr}{{\mathrm {Tr}}}

\newcommand{\F}{\mathscr{F}}

\newcommand{\A}{\mathscr{A}}
\newcommand{\Pc}{\mathscr{P}}

\newcommand{\vm}{{\bf m}}

%

%Use this to keep equation numbers the same size when using \small around equations
\makeatletter
\renewcommand{\maketag@@@}[1]{\hbox{\m@th\normalsize\normalfont#1}}%
\makeatother

\newtheorem{theorem}{Theorem}[section]

\newtheorem{lemma}[theorem]{Lemma}

\newtheorem{proposition}[theorem]{Proposition}

\newtheorem{definition}{Definition}[section]

% %use this to add email
% \catcode`\_=11\relax
%     \renewcommand\email[1]{\_email #1\q_nil}
%     \def\_email#1@#2\q_nil{%
%       \href{mailto:#1@#2}{{\emailfont #1\emailampersat #2}}
%     }
%     \newcommand\emailfont{}
%     \newcommand\emailampersat{\small@}
%     \catcode`\_=8\relax 
    
%\renewcommand{\thesection}{\Roman{section}}

%larger parentheses than \Bigg
\makeatletter
\newcommand{\vast}{\bBigg@{3.5}}
\newcommand{\Vast}{\bBigg@{4}}
\newcommand{\vastt}{\bBigg@{4.5}}
\newcommand{\Vastt}{\bBigg@{5}}
\makeatother

%Use this command to restart counting for equations in the supplement
\newcommand{\beginsupplement}{
        \setcounter{table}{0}
        \renewcommand{\thetable}{S\arabic{table}}
        \setcounter{figure}{0}
        \renewcommand{\thefigure}{S\arabic{figure}}
        \setcounter{section}{0}
        \renewcommand{\thesection}{S\arabic{section}}
        \setcounter{equation}{0}
        \renewcommand{\theequation}{S\arabic{equation}}
        \setcounter{algorithm}{0}
        \renewcommand{\thealgorithm}{S\arabic{algorithm}}
     }

\setlist[enumerate]{
  label={\upshape(\roman*)},
  leftmargin=30pt,
  labelwidth=0pt,
  align=right,
  nosep
}

%% file: Acronyms.tex
% \begin{acronym}
% \acro{DNN}{Deep Neural Network}
% \acro{ODE}{Ordinary Differential Equation}
% \acro{SPDE}{Stochastic Partial Differential Equation}
% \acro{FNN}{Feed-forward Neural Network}
% \acro{CNN}{Convolutional Neural Network}
% \acro{DP}{Dynamic Programming}
% \acro{LSTM}{Long-Short Term Memory}
% \acro{FC}{Fully Connected}
% \acro{DDP}{Differential Dynamic Programming}
% \acro{HJB}{Hamilton-Jacobi-Bellman}
% \acro{PDE}{Partial Differential Equation}
% \acro{PI}{Path Integral}
% \acro{NN}{Neural Network}
% \acro{SOC}{Stochastic Optimal Control}
% \acro{RL}{Reinforcement Learning}
% \acro{MPC}{Model Predictive Control}
% \acro{IL}{Imitation Learning}
% \acro{RNN}{Recurrent Neural Network}
% \acro{DL}{Deep Learning}
% \acro{SGD}{Stochastic Gradient Descent}
% \acro{SDE}{Stochastic Differential Equation}
% \acro{VRL}{Variational Reinforcement Learning}
% \acro{IDVRL}{Infinite Dimensional Variational Reinforcement Learning}
% \acro{1D}{1-dimensional}
% \acro{2D}{2-dimensional}
% \acro{ROM}{Reduced Order Model}
% \acro{SDE}{Stochastic \ac{ODE}}
% \acro{ANN}{Artificial Neural Network}
% \acro{ADPL}{Actuator Design and Policy Learning}
% \end{acronym}

\makeglossaries
\newacronym{SPDE}{SPDE}{Stochastic Partial Differential Equation}
\newacronym{ODE}{ODE}{Ordinary Differential Equation}
\newacronym{PDE}{PDE}{Partial Differential Equation}
\newacronym{SDE}{SDE}{Stochastic Differential Equation}
\newacronym{POD}{POD}{Proper Orthogonal Decomposition}
\newacronym{MOR}{MOR}{Model Order Reduction}
\newacronym{MPC}{MPC}{Model Predictive Control}
\newacronym{RMSE}{RMSE}{Root Mean Squared Error}
\newacronym{HJB}{HJB}{Hamilton-Jacobi-Bellman}
\newacronym{ANN}{ANN}{Artificial Neural Network}